\documentclass[11pt,reqno, english,a4paper]{amsart}

\usepackage[utf8x]{inputenc}
\usepackage{babel}
\usepackage[babel]{csquotes}

\usepackage{url}

\usepackage{amssymb}
\usepackage{amscd}
\usepackage{amsfonts}
\usepackage{mathrsfs}
\usepackage{setspace}
\usepackage{version}
\usepackage[pdftex,colorlinks]{hyperref}

\numberwithin{equation}{section}

\theoremstyle{plain}
\newtheorem{theorem}{Theorem}[subsection]
\newtheorem{proposition}[theorem]{Proposition}
\newtheorem{lemma}[theorem]{Lemma}
\newtheorem{corollary}[theorem]{Corollary}

\newcounter{dummy} \numberwithin{dummy}{section}
\newtheorem{thm}[dummy]{Theorem}
\newtheorem{cor}[dummy]{Corollary}
\newtheorem{lem}[dummy]{Lemma}
\newtheorem{prop}[dummy]{Proposition}
\newtheorem{conj}[dummy]{Conjecture}

\theoremstyle{definition}
\newtheorem{definition}[theorem]{Definition}
\newtheorem{example}[theorem]{Example}
\newtheorem{defn}[dummy]{Definition}
\newtheorem{ex}[dummy]{Example}

\theoremstyle{remark}
\newtheorem{remark}[theorem]{Remark}
\newtheorem*{rmk}{Remark}

\newcommand{\eps}{\varepsilon}

\def\g{\mathfrak{g}}

\newcommand\Z{\mathbb{Z}}
\newcommand\R{\mathbb{R}}

\newcommand\C{\mathbb{C}}
\newcommand\N{\mathbb{N}}

\newcommand\SL{\operatorname{SL}}

\newcommand\GL{\operatorname{GL}}

\newcommand\SU{\operatorname{SU}}

\newcommand\rk{\operatorname{rk}}

\newcommand\Supp{\operatorname{Supp}}

\newcommand\Hom{\operatorname{Hom}}

\newcommand\Grass{\operatorname{Grass}}

\newcommand\Q{\mathbb{Q}}

\def\g{\mathfrak{g}}

\newcommand{\uu}{\mathfrak{u}}

\newcommand{\bbB}{\mathbf}

\newcommand{\bg}{{\bbB g}}
\newcommand{\bX}{{\bbB X}}

\newcommand{\cB}{{\mathcal B}}

\newcommand{\cF}{{\mathcal F}}
\newcommand{\cW}{\mathcal{W}}

\newcommand{\cH}{{\mathcal H}}
\newcommand{\cL}{{\mathcal L}}
\newcommand{\cM}{{\mathcal M}}

\newcommand{\cP}{{\mathcal P}}

\newcommand{\cV}{{\mathcal V}}

\newcommand{\bp}{{\mathbf p}}

\newcommand{\bw}{{\mathbf w}}

\newcommand{\sx}{{\mathsf x}}

\newcommand{\sr}{{\mathsf r}}

\DeclareMathOperator{\sign}{sign}
\DeclareMathOperator{\Span}{Span}
\DeclareMathOperator{\dist}{dist}
\DeclareMathOperator{\Aff}{Aff}
\DeclareMathOperator{\Gal}{Gal}

\DeclareMathOperator{\Aut}{Aut}

\DeclareMathOperator{\vol}{vol}
\DeclareMathOperator{\diag}{diag}

\newcommand\set[1]{\left\{#1\right\}}
\newcommand\pa[1]{\left(#1\right)}

\begin{document}

\title{Diophantine approximation on matrices and Lie groups}
\author[M. Aka, E. Breuillard, L. Rosenzweig, N. de Saxc\'e]{Menny Aka, Emmanuel Breuillard, Lior Rosenzweig,\\ Nicolas de Saxc\'e}

\address{M.A. Departement Mathematik\\
ETH Z\"urich\\
R\"amistrasse 101\\
8092 Zurich\\
Switzerland}
\email{menashe-hai.akka@math.ethz.ch}

\address{Mathematisches Institut\\
62 Einsteinstrasse, Universit\"at M\"unster\\
M\"unster\\
Germany}
\email{emmanuel.breuillard@uni-muenster.de}

\address{Department of mathematics\\
 ORT Braude College\\
P.O.Box 78, Karmiel 21982\\
 Israel}
\email{liorr@braude.ac.il}

\address{CNRS -- Université Paris-Nord, LAGA\\
93430 Villetaneuse\\
France}
\email{desaxce@math.univ-paris13.fr}

\begin{abstract}We study the general problem of extremality for metric diophantine
approximation on submanifolds of matrices. We formulate a
criterion for extremality in terms of a certain family of algebraic obstructions
and show that it is sharp. In general the almost sure diophantine
exponent of a submanifold is shown to depend only on its Zariski closure,
and when the latter is defined over $\Q$, we prove that the exponent
is rational and give a method to effectively compute it. This method
is applied to a number of cases of interest. In particular we prove that the diophantine exponent of rational
nilpotent Lie groups exists and  is a rational number, which we determine explicitly in terms of representation theoretic data.
\end{abstract}


\maketitle
\setcounter{tocdepth}{1}

\tableofcontents

\section{Introduction}

Pick $n+m$ vectors $x_1,\ldots,x_{n+m}$ at random in $\R^m$ according to a certain distribution $\nu$. Take integer linear combinations of the vectors and ask how close they can get to the origin in $\R^m$. One way to measure this is via the diophantine exponent, defined as:

\begin{equation}\label{diophexp-def}\beta(x)= \inf\{\beta>0 ; \|\sum_{i=1}^{m+n}p_ix_i\|> \|\bp\|^{-\beta} \textnormal{ for all but finitely many } \bp \in \Z^{m+n}\},\end{equation}
where $x=(x_1,\ldots,x_{m+n})$, $\bp=(p_1,\ldots,p_{m+n})$ and where on both sides $\|\cdot\|$ is the supremum norm on the coordinates.

In the case where $x_{1}=e_1,\dots,x_{m}=e_m$, where $e_1,\ldots,e_m$ is the canonical basis of $\R^m$, and $(x_{m+1},\dots,x_{m+n})$ is chosen on a submanifold of $(\R^{m})^{n}\simeq M_{m,n}(\R)$, this question is the topic of \emph{Diophantine approximation on submanifolds of matrices}, which studies the quality of approximation by integer vectors of the image of an integer vector under an $m \times n$ matrix chosen at random in some submanifold of $M_{m,n}(\R)$.

When $m=1$, this subject has been studied extensively, starting with a 1932 conjecture of Mahler \cite{mahler}, which posited that the Veronese curve $$\mathcal{M} = \{(x, x^2, \ldots,  x^n) , \textnormal{ } x \in \R\}$$ is extremal, i.e. that a random point on it has the same diophantine exponent as a random point in $\R^n$ chosen with respect to Lebesgue measure. This conjecture was proved by Sprind\v{z}uk  \cite{sprindzuk1,sprindzuk2}. We refer the reader to \cite{sprindzuk1,sprindzuk2, dodsonbook, beresnevich-velani-survey} for background on Diophantine approximation on manifolds.

About twenty years ago, Kleinbock and Margulis \cite{kleinbock-margulis} revolutionized the subject by introducing methods from the dynamics of homogeneous flows into this area. In particular they settled a conjecture of Sprind\v{z}uk by establishing that every analytic submanifold of $\R^n$ that is not contained in a proper affine subspace is extremal (see \cite{beresnevich-bernik} for the state of the art before their work).  This method is based on certain \emph{quantitative non-divergence estimates} for diagonal flows on the space of lattices and takes its roots in the early work of Margulis \cite{margulis} pertaining to the non-divergence of unipotent flows on homogeneous spaces in connection with Margulis' first proof of arithmeticity for non-uniform higher rank lattices. It was later further developed by Kleinbock in an important series of papers \cite{kleinbock-duke, kleinbock-extremal-gafa, kleinbock-anextension, kleinbock-dichotomy}. We refer the reader to \cite{kleinbock-margulisbirthday, kleinbock-pisa} for a nice introduction to these techniques.

\bigskip

Our first goal in this article is to answer a question of Beresnevich, Kleinbock and Margulis \cite{beresnevich-kleinbock-margulis} asking for the right criterion for a submanifold of matrices to be extremal. We were led to this problem after we observed that a solution would enable us to compute diophantine exponents for dense subgroups of nilpotent Lie groups, in the spirit of our previous work \cite{abrs}. Beyond extremality, our criterion allows us to effectively compute the exponent of any rationally defined submanifold of matrices. Further in the paper, we will derive consequences for nilpotent Lie groups. The results of this paper were announced in \cite{abrsCRAS}.

\bigskip
  \centerline{*}
  \centerline{* *}
  \bigskip

We view the $(n+m)$-tuple $x$ as a $m \times (m+n)$ matrix. We note in passing that the exponent $\beta(x)$ depends only on the kernel $\ker x$. Therefore diophantine approximation on submanifolds of matrices is best phrased in terms of diophantine approximation on submanifolds of the grassmannian, here the grassmannian of $n$-planes in $\R^{n+m}$. We will keep this observation as a guiding principle, but will always phrase everything in terms of matrices. We now describe a family of obvious obstructions to extremality.

\begin{defn}[Pencil]\label{pencildef} Given a real vector subspace $W \leq \R^{m+n}$, and an integer $r\leq m$, we define the \emph{pencil} $\mathcal{P}_{W,r}$ to be the set of matrices $x \in M_{m,m+n}(\R)$ such that $$\dim (xW) \leq r.$$ We say that the pencil is rational if $W$ is rational, i.e. admits a basis in~$\Q^{m+n}$.
\end{defn}

\begin{rmk}
In the $m=1$ case, $M_{1,1+n}(\R)$ can be identified with the space of linear forms on $\R^{1+n}$.
For $W\leq\R^{1+n}$, the pencil $\cP_{W,0}$ is the subspace of linear forms vanishing on $W$.
Of course, for $r\geq 1$ we have $\cP_{W,r}=M_{1,1+n}(\R)$.
In the general case, pencils are closed algebraic subvarieties of $M_{m,m+n}(\R)$, but they need not be affine subspaces.
\end{rmk}

Pencils are obvious obstructions to extremality, because if $W$ is rational, then the standard Dirichlet or pigeonhole argument shows that for every $x \in \mathcal{P}_{W,r}$
$$1+ \beta(x) \geq \frac{\dim W}{r}.$$
Indeed there are roughly $R^{\dim W}$ integer points in the ball of radius $R$ in $W$, but they get mapped into a ball of radius $O(R)$ in $\R^r$, so, comparing volumes, $\eps$-balls around each of these points cannot be all disjoint if $\eps$ is at least of order $R^{1-\dim W /r}$. Consequently, we see that if \begin{equation}\label{constrain}\frac{\dim W}{r} > \frac{m+n}{m}\end{equation}
then the point $x$ is not extremal, since $\frac{n}{m}$ is the diophantine exponent of a random point of $M_{m,m+n}(\R)$ chosen with respect to Lebesgue measure.

Pencils satisfying $(\ref{constrain})$  will be called \emph{constraining pencils}, because they constrain the diophantine exponent away from its extremal value. A pencil is called \emph{proper} if it is not all of $M_{m,m+n}(\R)$. When viewed in the grassmannian (i.e. looking at the set of kernels $\ker x$, with $x$ in a pencil), pencils are well-studied objects: they are a certain kind of \emph{Schubert varieties} (see e.g. \cite[ch 1.5]{griffiths-harris} and Section~\ref{schubert} below).

When the Zariski closure of $\mathcal{M}$ is defined over $\Q$, it turns out that these obvious obstructions are the only ones, and the following is our first main result.

\begin{thm}[Exponent for rational manifolds]\label{rational-exponent-babythm} Let $m,n$ be positive integers and $\mathcal{M}$ be a connected analytic submanifold of $M_{m,m+n}(\R)$. Assume that the Zariski closure of $\mathcal{M}$ is defined over $\Q$. Then for Lebesgue almost every point $x \in \mathcal{M}$ the diophantine exponent is rational and equals
\begin{equation}\label{formulabeta}\beta(x) = \max_{\mathcal{M} \subset \mathcal{P}_{W,r}}  \{\frac{\dim W}{r} -1\}.\end{equation}
Moreover, the maximum in the right-hand side is achieved for a rational $W$.
\end{thm}

Note that the maximum in \eqref{formulabeta} is always achieved, because the quantity $\frac{\dim W}{r}-1$ takes values in a finite set of rational points;
the important point in the second part of the theorem is that some \emph{rational} subspace achieves the maximum.

The theorem shows in particular that the generic value of $\beta(x)$ is constant, and can be effectively computed once the pencils containing $\mathcal{M}$ are identified. We also note that this value is the smallest possible value for the exponent of an arbitrary point on $\mathcal{M}$, because of the Dirichlet or pigeonhole type argument explained above.

When speaking of Zariski closure, algebraic subsets and algebraic varieties in this paper, we will always consider these notions in real algebraic geometry and we refer the reader to the textbook \cite{bochnak-coste-roy} for definitions and basic properties. By Lebesgue measure on $\mathcal{M}$ we mean the top dimensional Hausdorff measure of the subset $\mathcal{M}$ of $M_{m,m+n}(\R)$. The real algebraic variety $\cM$ is said to be defined over $\Q$ if it is the set of zeroes of a family of polynomials (in the matrix entries) with rational coefficients. We immediately conclude:

\begin{cor}\label{corollary-extrem}
Let $\mathcal{M}$ be as in Theorem \ref{rational-exponent-babythm}. Then $\mathcal{M}$ is extremal if and only if it is not contained in any constraining pencil.
\end{cor}

Theorem \ref{rational-exponent-babythm} and its proof hold verbatim for more general measures (than Lebesgue measure on an analytic submanifold), which we call locally good measures (see Definition \ref{locally-good-def}).

The above theorem answers a question discussed at the end of the original paper of Kleinbock and Margulis \cite[\S 6.2]{kleinbock-margulis} and also raised in the problem list \cite[\S 9.1]{gorodnik-AIM} as well as in \cite[p.23]{kleinbock-margulis-wang} and \cite[Problem 1]{beresnevich-kleinbock-margulis}. The papers \cite{kleinbock-margulis-wang}  and \cite{beresnevich-kleinbock-margulis} proposed other sufficient criterions for extremality of an analytic submanifold of $M_{m,n}(\R)$, but they just failed to be optimal. The weak non-planar condition of \cite{beresnevich-kleinbock-margulis} is not strong enough for example to show extremality in the applications to nilpotent groups described below. This condition is equivalent in our language to not being contained in any proper pencil, which is a strictly stronger condition (see Example \ref{BKM-ctex}).

\bigskip

We now no longer assume the Zariski closure to be defined over $\Q$. In three papers \cite{kleinbock-extremal-gafa}, \cite{kleinbock-anextension} and \cite{kleinbock-dichotomy} Kleinbock studied this situation. He showed in particular that the diophantine exponent of a random point on a connected analytic submanifold of $M_{m,n}(\R)$ achieves the same value almost everywhere. In the case where $m=1$ he showed that this almost sure value depends only on the affine span of the submanifold in $\R^n$. In other words a submanifold inherits its exponent from its affine span. Another remarkable observation he made in \cite{kleinbock-anextension} for $\R^n$ and in  \cite[Theorem 1.4]{kleinbock-dichotomy} for matrices, is that the diophantine exponent of a random point is also the worst diophantine exponent of any point on the submanifold. In the context of matrices, one needs to find the right replacement for the affine span. We show the following:

\begin{thm}[Inheritance from Pl\"ucker closure]\label{pluckerklein} Let $\mathcal{M}$ be a connected analytic submanifold of $M_{m,m+n}(\R)$. The generic value of the diophantine exponent $\beta(x)$ depends only on the Pl\"ucker closure  $\mathcal{H}(\mathcal{M})$ of $\mathcal{M}$. In particular it depends only on the Zariski closure of $\mathcal{M}$ in $M_{m,m+n}(\R)$. Moreover for Lebesgue almost every point $x \in \mathcal{M}$
\begin{equation}\label{klein} \beta(x) = \inf_{y \in \mathcal{M}}\beta(y)
\end{equation}
\end{thm}

The existence of a generic value for $\beta(x)$, which is the same for almost every point $x$ of $\mathcal{M}$, and the identity $(\ref{klein})$ are due to Kleinbock \cite[Theorem 1.4.a]{kleinbock-dichotomy}. When we say that the generic exponent depends only on the Pl\"ucker or on the Zariski closure of $\cM$, this means in particular that the generic value of the exponent of a random point of $\mathcal{H}(\cM)$ or of the Zariski closure of $\cM$ (considered with respect to the Lebesgue measure on these algebraic varieties) coincides with that of $\cM$. 

The Pl\"ucker closure of a subset $S$ in $M_{m,m+n}(\R)$ is the set of all matrices $x$ whose full list of  minors satisfy the same linear equations as those of $S$. It contains the Zariski closure of $S$ and is contained in the Schubert closure of $S$, namely the intersection of all pencils containing $S$. See \S \ref{subsecquant}.


Theorem \ref{pluckerklein} has been established independently in the recent work of Das, Fishman, Simmons and Urbanski, see \cite[Theorem 1.9]{das-fishman-simmons-urbanski}. In this work the authors introduce a very general class of measures, which is invariant under measure automorphisms and encompasses many fractal measures of dynamical origin, for which Theorem \ref{pluckerklein} is shown to hold, see the discussion after \cite[Definition 1.2]{das-fishman-simmons-urbanski}.

Our Theorems \ref{rational-exponent-babythm} and \ref{general-inequality} will be proved in Section~\ref{quasinorm-sec} for a smaller class of measures on $\Hom(V,E)$ (still more general than Lebesgue measure on analytic submanifolds), which we call \emph{locally good measures} and are basically the matrix analogue of the friendly measures introduced by Kleinbock, Lindenstrauss and Weiss in \cite{kleinbock-lindenstrauss-weiss}.

We also stress that all of our results will be proved for submanifolds of $\Hom(V,E)$, where $V$ and $E$ are finite-dimensional real vector spaces of arbitrary dimension (we will not assume $\dim V> \dim E$ as in this introduction).

We further show the following general inequality:

\begin{thm}[Bounds for the exponent]\label{general-inequality} Let $\mathcal{M}$ be a connected analytic submanifold of $M_{m,m+n}(\R)$. Then for almost every $x \in \mathcal{M}$ with respect to Lebesgue measure:
\begin{equation}\label{ratformula}
\max_{\mathcal{M} \subset \mathcal{P}_{W,r} ; W \textnormal{ rational}}  \{\frac{\dim W}{r} -1\} \leq \beta(x) \leq  \max_{\mathcal{M} \subset \mathcal{P}_{W,r}}  \{\frac{\dim W}{r} -1\}.
\end{equation}
\end{thm}

The maximum is taken over rational pencils on the left-hand side and arbitrary pencils on the right-hand side. In particular:

\begin{cor}\label{corextremal}
If $\mathcal{M}$  is not contained in any constraining pencil, then $\mathcal{M}$ is extremal.
\end{cor}

When the Zariski closure is defined over $\Q$, Theorem~\ref{rational-exponent-babythm} reduces the determination of the exponent to a purely algebraic question: determining the maximum in $(\ref{formulabeta})$. This may still be a challenging task, however the analysis is greatly simplified by the following property of the pencils realizing the maximum in $(\ref{formulabeta})$.

\begin{prop} There is a unique subspace $W\leq \R^{m+n}$ of maximal dimension such that $\mathcal{M} \subset \mathcal{P}_{W,r}$ and $$\frac{\dim W}{r}=\max_{\mathcal{M} \subset \mathcal{P}_{W,r}} \frac{\dim W}{r} .$$
\end{prop}

This fact is a consequence of a very general lemma, which we call the \emph{submodularity lemma} and runs as follows.

\begin{lem}[Submodularity lemma]\label{submodularity}
Let $V$ be a finite-dimensional vector space, and suppose that $\phi:\Grass(V)\rightarrow\N$ is a non-decreasing  (for set inclusion) and \emph{submodular} function, i.e. for any two vector subspaces $U$ and $W$ we have
\[ \phi(U+ W)+\phi(U\cap W) \leq \phi(U)+\phi(W). \]
Then the maximum
\[ \max_{W\in\Grass(V)\setminus\{0\}} \frac{\dim W}{\phi(W)} \]
is attained at a unique subspace of maximal dimension.
\end{lem}

In particular this applies to the right-hand side in $(\ref{ratformula})$, because the function $W \mapsto \max_{x \in \cM} \dim xW$ is submodular. 

In presence of symmetry, this lemma greatly simplifies the algebraic analysis of determining the  right-hand side in $(\ref{ratformula})$ and hence the exponent. Indeed if $\phi$ is $G$-invariant for the action of some group $G$ on the grassmannian (preserving dimension), then the submodularity lemma implies that the subspace realizing the maximum is $G$-invariant. This observation is used to derive Theorem \ref{rational-exponent-babythm} from Theorem \ref{general-inequality} by means of the Galois group action on $V$. Here is an example illustrating the use of the submodularity lemma for another group action in combination with Theorem \ref{rational-exponent-babythm} :

\begin{ex}[The Veronese manifold in matrices] Let $p,s \in \N$. What is the infimum $\widehat{\beta}$ of all $\beta>0$ such that for almost every $M\in M_s(\R)$, the inequality
\[ \|v_0 I+ v_1 M + \dots + v_p M^p\| \leq \|v\|^{-\beta} \]
has at most finitely many integer solutions $v\in\Z^{p+1}$ ? The Mahler conjecture proved by Sprind\v{z}uk mentioned earlier is the case $s=1$, and then the answer is $\widehat{\beta} = p$.
When $s>1$, this problem fits the diophantine approximation on submanifolds of matrices scheme: set $V=\R^{m+n}=\R_p[X]$ the space of polynomials of degree at most $p$.
Naturally, $\Z^{m+n}$ is identified with the lattice of polynomials with integer coefficients.
Set $E=\R^m=M_s(\R)$, and $\mathcal{M} \subset \Hom(V,E)\simeq M_{m,m+n}(\R)$ the set of all linear maps $P \mapsto P(M)$.
In terms of matrices, $\cM$ is just the Veronese manifold $\{(I,M,\ldots,M^p) \ ;\ M\in M_s(\R)\}$.
Using Theorem \ref{rational-exponent-babythm} the exponent $\widehat{\beta}$ can be determined once the pencils containing $\mathcal{M}$ are determined. But the group $G$ of affine transformations of the real line acts on $V$ by substitution of the variable and preserves $\mathcal{M}$. So in order to compute the maximum in $(\ref{rational-exponent-babythm})$, we only need to consider $G$-invariant subspaces. This is easily done, and we find that
$$1+ \widehat{\beta} = \max\{\frac{p+1}{s} , 1\}.$$
In particular the Veronese manifold is extremal if and only if $s \geq p+1$. In this example the most constraining pencil is given by the Cayley-Hamilton relation. See Example  \ref{veronese} for more details. See also \cite[\S 3.3]{kleinbock-dichotomy} and \cite{das-simmons}, where the Veronese manifold is considered with respect to a different diophantine approximation scheme, for which it is always extremal.
\end{ex}

Of course when the Zariski closure is not defined over $\Q$, it can happen that the almost sure exponent is not given by either side of $(\ref{ratformula})$.  Then the question remains to find an appropriate hull $S(\mathcal{M})$ of $\mathcal{M}$ whose almost sure exponent would be the same as that of $\mathcal{M}$, in the spirit of Kleinbock's work \cite{kleinbock-extremal-gafa} \cite{kleinbock-anextension} on submanifolds of $\R^n$ with irrational affine span. The natural candidate is what we call the Schubert closure of $\mathcal{M}$, which is the intersection of the pencils containing $\mathcal{M}$. It is an algebraic variety, which is in general bigger than the Zariski closure of $\mathcal{M}$. In the classical setting of submanifolds of $M_{m,n}(\R)$ this is the same space as the space $\mathcal{H}(\mathcal{M})$ considered by Beresnevich, Kleinbock and Margulis in \cite[7.2]{beresnevich-kleinbock-margulis}. The Pl\"ucker closure considered here and in  \cite{das-fishman-simmons-urbanski}  is in general smaller than the Schubert closure. 

The results above are a first stone in a more complete study of diophantine approximation on submanifolds of matrices that we do not undertake here and would be concerned with badly approximable points (such as in \cite{baker, kleinbock-bad}), improvements of Dirichlet's theorem (see \cite{kleinbock-weiss-dirichlet}), Khintchine-type theorems (in the spirit of \cite{bernik-kleinbock-margulis, beresnevich-bernik-kleinbock-margulis}), Jarnik-type theorems, etc. 
Let us only mention that the methods introduced in the current paper also apply to the study of multiplicative diophantine approximation as in the original work of Kleinbock and Margulis \cite{kleinbock-margulis} ; in particular, one can derive a satisfying criterion for a manifold to be strongly extremal. This will be explained in a forthcoming paper of the first and last authors with Das and Simmons \cite{bdss}.

\bigskip

In order to apply the above results to the computation of diophantine exponents of random subgroups of nilpotent Lie groups, it is necessary (in the general case) to extend them to the setting of \emph{quasi-norms} and consider \emph{weighted diophantine approximation}, in which different directions are assigned possibly different weights. We now briefly describe what this means, but later in the paper we will prove our results in this generalized setting.

As above $\mathcal{M}=\Phi(U)$ is a connected analytic submanifold of $\Hom(V,E)$, where $V$ and $E$ are two finite-dimensional real vector spaces, $U \subset \R^N$ is a connected open set and $\Phi:U \to \Hom(V,E)$ an analytic map. We endow $V$ and $E$ with quasi-norms, namely
$$|v| =\max|v_i|^{\frac{1}{\alpha_i}}$$
$$|w|' =\max |w_i|^{\frac{1}{\alpha'_i}},$$
for $v=\sum_i v_i u_i\in V$ and $w=\sum_i w_i u'_i\in E$, where $(u_1,\ldots,u_d)$ is a basis of $V$ and $(u'_1,\ldots,u'_e)$ a basis of $E$, and $\alpha_1\geq \ldots \geq \alpha_d>0$ and $\alpha'_1\geq \ldots \geq \alpha'_e>0$ are positive numbers. Next we let $\Delta$ be the lattice $\Z u_1 + \ldots +\Z u_d$ and we consider the following diophantine exponent for $x \in \Hom(V,E)$
\begin{equation}\label{diophquasi}
\beta(x):= \inf \{ \beta>0 \ |\ \exists c>0:\,\forall v\in\Delta\setminus\{0\},\, |xv|' > c |v|^{-\beta}\}.
\end{equation}

Given two non-negative numbers $a,b$, and a subspace $W$ of $V$, we define (see Section \ref{penpen}) the \emph{pencil}
$$\mathcal{P}_{W,a,b} =\{x \in \Hom(V,E) \ |\ \psi(\ker x \cap W) \geq a \textnormal{ and }\phi(xW) \leq b\},$$ where $\psi$ (resp. $\phi$) is defined for a subspace $W\leq V$ (resp. $W'\leq E$) by
$$\psi(W)=\sum_{i \in I(W)}  \alpha_i
\quad\mbox{and}\quad \phi(W')=\sum_{i \in J(W')}  \alpha'_i,$$
and where
\[
I(W)=\{i\in\{1,\ldots,d\} \ |\ \dim(W \cap \langle u_1, \ldots,u_i\rangle) > \dim(W \cap \langle u_1, \ldots,u_{i-1}\rangle)\}
\]
and
\[
J(W')=\{i\in\{1,\ldots,e\} \ |\ \dim(W' \cap \langle u'_{i+1}, \ldots,u'_e\rangle) < \dim(W' \cap \langle u'_i, \ldots,u'_e\rangle)\}.
\]
The pencil $\mathcal{P}_{W,a,b}$ is a certain closed algebraic subset of $\Hom(V,E)$ associated to our choice of quasi-norms on $V$ and $E$, and coincides with the pencil defined in Definition \ref{pencildef} in the unweighted case. Pencils are indeed closely related to Schubert subvarieties of the grassmannian $\Grass(V)$ (see Section \ref{schubert}).

\begin{thm}[Weighted case]\label{weighted} Theorems \ref{rational-exponent-babythm} and \ref{general-inequality} hold more generally in the weighted (quasi-norm) setting. In particular for Lebesgue almost every $x \in U$,
\begin{equation}\label{ratformula2}
\max_{\mathcal{M} \subset \mathcal{P}_{W,a,b} ; W \textnormal{ rational}}  \frac{a}{b} \leq \beta(\Phi(x)) \leq  \max_{\mathcal{M} \subset \mathcal{P}_{W,a,b}}  \frac{a}{b}.
\end{equation}
with equality when the Zariski closure of $\cM$ is defined over $\Q$.
\end{thm}

We refer the reader to Theorems \ref{uppboundnud} and \ref{uppboundnudrat} for more complete statements. The $\Q$-structure on $\Hom(V,E)$ used implicitly in the above theorem is the one induced by the $\Q$-span of the bases of $V$ and $E$ chosen to define the quasi-norms. The condition that the Zariski closure of $\cM$ be defined over $\Q$ is satisfied for example when the map $\Phi$ itself is a polynomial map with coefficients in $\Q$ (i.e. each matrix entry is such). 

The dynamical method adapts well to the weighted case, as had already been observed by Kleinbock early on, for example in \cite{kleinbock-duke}.

\bigskip

\bigskip
  \centerline{*}
  \centerline{* *}
  \bigskip

We now pass to the description of our results regarding diophantine approximation on nilpotent Lie groups. Inspired by the work of Gamburd, Jakobson and Sarnak \cite{gamburd-jacobson-sarnak} we introduced in our previous paper \cite{abrs} the notion of diophantine Lie group. We refer the reader to this paper for an introduction and background on the general question of diophantine approximation on Lie groups. We briefly recall here the basic definitions. We are given a connected Lie group $G$ and a $k$-tuple of elements $\mathbf{g}:=(g_1,\ldots,g_k)$. We consider the subgroup $\Gamma_{\bg}$ generated by $g_1,\ldots,g_k$. Its elements can be represented by words $w(\bg)$ in the elements $g_i$. We ask: How close to the identity can an element $\gamma$ of $\Gamma_{\bg}$ be in terms of the minimal length $\ell(\gamma)$ of a word that represents it ? The following definition is natural. Let $d(x,y)$ be a left-invariant Riemannian distance on $G$ (or more generally a left-invariant geodesic metric) and let  $V_{\bg}(n)$ be the number of elements in $\Gamma_{\bg}$ that can be represented by a word of length at most $n$.

\begin{defn} The $k$-tuple $\bg=(g_1,\ldots,g_k)\in G^k$ (or the subgroup $\Gamma_{\bg}$ they generate) is called \emph{diophantine} if there exists $\beta>0$ such that \begin{equation}\label{groupdioph}d(\gamma,1) > V_{\bg}(\ell(\gamma))^{-\beta}\end{equation}
for all but at most finitely many group elements $\gamma \in \Gamma_{\bg}$.
\end{defn}

For example in $G=(\R,+)$ the pair $\bg=(1,\alpha)$ is diophantine if and only if $\alpha$ is a diophantine, i.e. non-Liouville, number. It is easy to check that this definition depends only on the subgroup $\Gamma_{\bg}$ and not on the particular generating set $\bg$, nor on the choice of metric. We say that the Lie group $G$ is \emph{diophantine on $k$-letters} if almost every $k$-tuple in $G$ (chosen independently at random with respect to Haar measure) is diophantine. We say that it is \emph{diophantine} if it is diophantine on $k$-letters for every $k\geq 1$. When $G=\R^n$, then $\Gamma_{\bg}$ is diophantine if and only if the $n \times k$ matrix whose column vectors are $g_1,\ldots,g_k$ is diophantine in the sense that its diophantine exponent (defined as in $(\ref{diophexp-def})$) is finite. Hence the connection with diophantine approximation on matrices.

It is conjectured that semisimple real Lie groups are diophantine. This conjecture is open already for the smallest Lie groups, such as $SO(3,\R)$ (see \cite[\S 6]{hoory-gamburd}).  Remez-type inequalities \cite{brudnyi-ganzburg} combined with the Borel Cantelli lemma only yield a superexponential lower bound $\exp(-C\ell(\gamma)^2)$ for Lebesgue almost every $\bg$ in a semisimple Lie group (see the work of Kaloshin and Rodnianski \cite{kaloshin-rodnianski} who handled $\SU(2)$ but the method is general). It is also already an open problem to show that the affine group $\Aff(\R)$ of the real line and the group of motions of the Euclidean plane $O(2)\ltimes \R^2$ are diophantine. See the work of Varj\'u \cite{varju} for a very interesting recent result in this direction. 

 It is fairly easy to see that any nilpotent Lie group with a rational structure (for its Lie algebra) is diophantine, but in \cite{abrs} we constructed examples (arising only in nilpotency class $6$ and higher) of non-diophantine nilpotent Lie groups. They exist  because of some particular feature of the representation theory of the general linear group $\GL_k$ on the free Lie algebra on $k$-letters: multiplicity for the $s$-th homogeneous part, $s \geq 6$.

For nilpotent Lie groups the growth function $V_{\bg}(n)$ grows polynomially like $n^\tau$ with an integer exponent $\tau$ given by the Bass-Guivarc'h formula (see \cite{bassformula,guivarchformula}), and therefore changing the generating set in $\Gamma_{\bg}$ does not result in a change of the exponent $\beta$. So it makes sense to ask for the optimal $\beta$ for which $(\ref{groupdioph})$ holds. We call this the exponent of the subgroup $\Gamma_{\bg}$.
This is the quantity we study here, with the help of Theorem \ref{rational-exponent-babythm} and Lemma \ref{submodularity}.

It turns out that the optimal exponent always exists.

\begin{thm}[Existence of the exponent]\label{existence} Let $G$ be a connected and simply connected nilpotent real Lie group endowed with a left-invariant geodesic metric $d$. Then for each $k\geq 1$, there is $\widehat{\beta}_k \in [0,+\infty]$ such that for almost every $k$-tuple $\bg \in G^k$ with respect to Haar measure, we have
$$\beta(\bg) = \widehat{\beta}_k,$$
where $\beta(\bg)$ is the infimum of all $\beta>0$ such that $(\ref{groupdioph})$ holds.
\end{thm}

While the property of being diophantine or not for a $k$-generated subgroup of $G$ does not depend on the choice of metric near the identity, the exponent does. By a geodesic metric (a.k.a length metric) we mean a distance that is defined in terms of the length of the shortest path between two points. It is well known that left-invariant geodesic metrics on connected Lie groups are all Carnot-Carathéodory-Finsler metrics induced by a norm on a generating subspace of the Lie algebra (this is Berestowski's theorem, \cite{berestowski}). Riemannian metrics are of course examples of such, but in the context of nilpotent Lie groups non Riemannian, Carnot-Carath\'eodory metrics are also very natural.

We will present two proofs of Theorem~\ref{existence}. The first is only valid when $k\geq \dim G/[G,G]$, and is based on a relatively simple general argument that relies on the ergodicity of the action of the group of automorphisms of the free Lie algebra on $k$-tuples of elements in $\g$. The second proof on the other hand works for all $k$ and necessitates to first translate the diophantine problem in terms of (weighted) diophantine approximation on submanifolds of matrices, and then use the techniques of homogeneous dynamics alluded above.

This translation will allow us to determine $\widehat{\beta}_k$ for all rational nilpotent Lie groups endowed with a rational left-invariant geodesic metric (e.g. a Riemannian metric, see \S\ref{ratsec} for the definition).
Our main result is the following:

\begin{thm}[Rationality and stability of the exponent]\label{rationality-stability}
Let $G$ be a connected and simply connected rational nilpotent real Lie group endowed with a rational left-invariant geodesic metric.
Then for every $k\geq 1$ the exponent  $\widehat{\beta}_k$ is rational.
Furthermore, there exist an integer $k_1$ and a rational function $F_{\g} \in \Q(X)$ with rational coefficients such that, for all $k\geq k_1$,
$$\widehat{\beta}_k = F_{\g}(k).$$
\end{thm}

The rational function $F_{\g}(X)$ is a ratio $\frac{P}{Q}$, where $P$ and $Q$ are both polynomials of the same degree with rational coefficients. The degree is equal to the nilpotency class of $G$.

The constant $k_1$ will be shown to depend only on $\dim \g$. The value of $\widehat{\beta}_k$ for small $k$ may not fit the general pattern  $F_{\g}(k)$. The fact that it does for $k$ large is an instance of the so-called \emph{representation stability}, as per the notion investigated by Church, Ellenberg and Farb in a series of recent papers. Here stability occurs for the action of $\GL_k$ on the free Lie algebra on $k$-letters, see \cite[Corollary 5.7]{church-farb} for the case of interest to us.

In \cite{abrs} we showed that a nilpotent Lie group $G$ with Lie algebra $\g$ is diophantine  on $k$-letters if and only if its Lie algebra of laws on $k$-letters $\mathcal{L}_{k,\g}$ is a diophantine subspace of the free Lie algebra $\mathcal{F}_k$ on $k$-letters endowed with its natural $\Q$-structure. A law of $\g$ is an element of the free Lie algebra that vanishes identically when the indeterminates are replaced with arbitrary elements of $\g$.

The rational function $F_{\g}(k)$ can be computed explicitly in terms of $\g$ and the representation theory of the $\GL_k$-action by linear substitutions on the relatively free Lie algebra $\mathcal{F}_k/\mathcal{L}_{k,\g}$. In Sections \ref{section:critical} and \ref{section:examples} we will give exact formulas in a number of examples. In particular, we compute $\widehat{\beta}_k$ for metabelian groups, for the group of unipotent upper-triangular matrices and for certain free nilpotent groups.

In \cite[\S 5.2]{abrs} we asked whether $\widehat{\beta}_k$ has a limit as $k$ tends to infinity. With the above tools it is possible to answer this question, in the case where the nilpotent Lie group $G$ is rational.

\begin{cor}[Rational limit]\label{theorem-limit} Under the assumptions of the previous theorem
$$\lim_{k \to +\infty} \widehat{\beta}_k$$
exists and is a rational number contained in the interval $[\frac{1}{s\dim G^{(s)}},1]$ (if $d(\cdot,\cdot)$ is Riemannian, the lower bound is $\frac{1}{\dim G^{(s)}}$).
\end{cor}
Here $G^{(s)}$ is the last step in the descending central series of $G$.
 
Although we again make the assumption here that the Lie algebra $\g$ is defined over $\Q$, we believe that the above limit is always rational without such a rationality assumption, provided of course that $G$ is diophantine:

\begin{conj}[Rationality conjecture]\label{conjconj} Suppose $G$ is a connected nilpotent real Lie group endowed with a left-invariant geodesic metric. Assume that $G$ is diophantine. Then the limit $$\lim_{k \to +\infty} \widehat{\beta}_k$$
exists and is a rational number.
\end{conj}

Recall that $G$ is said to be diophantine if it is diophantine on $k$-letters for every positive integer $k$. In \cite{abrs} we constructed for each integer $k$ a connected nilpotent Lie group that is diophantine on $k$ letters, but not on $k+1$ letters.

\bigskip
In order to apply the diophantine results for submanifolds of matrices expounded at the beginning of this introduction, we will apply Theorem \ref{weighted} with $U=\g^k$ and $V=\mathcal{F}_k/\mathcal{L}_{k,\g}$, $E=\g$ and $\Phi(x)$ the evaluation map at $x=(X_1,\ldots,X_k) \in \g^k$. The quasi-norm on $V$ will be defined in terms of the descending central series of $V$ and the parameters $\alpha_i$ will be integers. The choice of a left-invariant geodesic metric on $G$ will yield a quasi-norm on $E$ and Proposition \ref{words-brackets} will show that the diophantine problem $(\ref{groupdioph})$ translates precisely into $(\ref{diophquasi})$.  It turns out that the submodularity lemma is also of great help for this in the case of nilpotent groups. It will be used this time for the action of $\GL_k$ on $\mathcal{F}_k/\mathcal{L}_{k,\g}$. This will mean that once the representation theory of $\mathcal{F}_k/\mathcal{L}_{k,\g}$ viewed as a $\GL_k$-module is understood well enough, the diophantine exponent can be computed. We will do just that in Section \ref{section:examples} for a number of concrete examples, for example for free nilpotent groups and for metabelian nilpotent groups.

\bigskip

Part of the results proved in this paper were announced in \cite{abrsCRAS}. The paper is organized as follows.
In Section \ref{section:ergodicity} we give the first proof of Theorem \ref{existence} on the existence of the exponent via the ergodicity of the group of rational points of the group of automorphims of the free Lie algebra acting on $k$-tuples.
Section \ref{section:heisenberg} is devoted to an important example, the Heisenberg group: we prove Theorem \ref{rationality-stability} for this group by an ad hoc argument using a Remez-type inequality for quadratic forms. The limitation of this method for tackling the more general nilpotent groups shows the power of the method used later on. Section \ref{quasinorm-sec} is devoted to the Kleinbock-Margulis method and the Dani correspondence, which allows to reformulate the diophantine approximation problem in terms of orbits in the space of lattices.
The proof of Theorem~\ref{uppboundnud} is given in Section~\ref{quasinorm-sec}, and Theorem~\ref{uppboundnudrat} is derived in Section~\ref{section:submodular}, after the submodularity lemma.
In Section \ref{section:critical} we give a second proof of Theorem \ref{existence} and establish the rationality of the exponent -- Theorem \ref{rationality-stability} -- in full generality using the diophantine approximation results for submanifolds of matrices developed in the previous sections.
In Section \ref{section:examples} we compute the exponent explicitly in a number of examples using representation theory of the free Lie algebra.

\bigskip


\bigskip

\noindent \emph{Acknowledgements.} It is a pleasure to thank  V. Beresnevich and D. Kleinbock for interesting discussions in relation to the topics of this paper. We are also grateful to the referee for his comments. The first author acknowledges the support of ISEF, Advanced Research Grant 228304 from the ERC, and SNF Grant 200021-152819. The second author acknowledges support from ERC Grant no 617129 GeTeMo. The third author was supported by the G\"oran Gustafssons Stiftelse f\"or Naturvetenskaplig och Medicinsk Forskning and Vetenskapsradet (grant no. 621-2011-5498).

\section{A zero-one law}\label{section:ergodicity}

We now turn to the problem of diophantine approximation in nilpotent Lie groups.
The purpose of this section is to show the existence of a critical exponent for any simply connected nilpotent real Lie group (not necessarily defined over $\Q$).
Later, in Sections~\ref{section:heisenberg}, \ref{section:critical} and \ref{section:examples}, we will explain how to compute this exponent when the group is defined over the rationals.

In this section, $G$ denotes an arbitrary connected simply connected nilpotent Lie group endowed with a left-invariant distance $d(\cdot,\cdot)$ inducing the topology.
A finitely generated subgroup $\Gamma$ of $G$ will be called \emph{$\beta$-diophantine} if there is a symmetric generating set $S$ of $\Gamma$ and a constant $c>0$ such that for every integer $n$ (recall that $S^n$ denotes the set of products of $n$ elements from $S$),
\begin{equation}\label{betadioph}
\min_{x\in S^n\backslash\{1\}} d(x,1) \geq c\cdot |S^n|^{-\beta}.
\end{equation}
By the Bass-Guivarc'h formula \cite{bassformula,guivarchformula}, we know that, within positive multiplicative constants, $|S^n| \simeq n^\tau$ for some integer $\tau$ depending on $\Gamma$ only. It follows that if (\ref{betadioph}) holds for some generating set $S$, then it will also hold for any other generating set, possibly with a different constant $c$, but with the same $\beta$. We can therefore define the diophantine exponent $\beta(\Gamma)$ of a finitely generated subgroup of $G$ by
\begin{equation}\label{groupexponent}
\beta(\Gamma) = \inf\{\beta \ |\  \Gamma\ \mbox{is $\beta$-diophantine}\}.
\end{equation}

\begin{thm}[Zero-one law]\label{zerone}
Let $G$ be a simply connected nilpotent Lie group, whose abelianization has dimension $d$. Let $k \geq d$ be an integer.
Given $\beta\geq 0$ the set of $k$-tuples $\bg=(g_1,\dots,g_k)$ generating a group $\Gamma_{\bg}$ that is $\beta$-diophantine is either null or co-null for the Haar measure on $G^k$.
\end{thm}

In particular, there is a number $\widehat{\beta}_k \in [0,+\infty]$ such that
\[ \beta(\Gamma_{\bg}) = \widehat{\beta}_k
\quad\mbox{for almost every $k$-tuple}\ \bg\in G^k.\]
This shows that Theorem \ref{zerone} implies Theorem \ref{existence} from the introduction, in the case $k \geq d$. An alternative proof (including the cases $k <d$) will be given in Section~\ref{section:critical}.
We record here the following open problems:

\bigskip
\noindent (1) Is the set of $k$-tuples $\bg$ such that $\Gamma_\bg$ is $\widehat{\beta}_k$-diophantine null or co-null ? Consistency with the case $G=\R$, where it is known that badly approximable tuples have zero measure, hints that the answer ought to be null, provided $k>\dim G/[G,G]$.

\noindent (2) Is there a Jarník type theorem for $k$-tuples, i.e. a formula for the Hausdorff dimension of the set of $k$-tuples $\bg$ such that $\beta(\Gamma_\bg) >\beta$ for any given $\beta$ ?

\bigskip

The rest of this section is devoted to the proof of Theorem \ref{zerone}, which is based on an ergodicity argument.

\subsection{Ergodic action on $\g^k$ by rational automorphisms}

Let $\cF_{k,s}$ be the free $s$-step nilpotent Lie algebra on $k$ generators. The group $\Aut(\cF_{k,s})$ of linear automorphisms of $\cF_{k,s}$ is an algebraic group defined over $\Q$.
If $\alpha \in \Aut(\cF_{k,s})(\R)$ and $\sx_1,\ldots, \sx_k$ are free generators of $\cF_{k,s}$, then for each $i=1,\ldots,k$, we let $\alpha_i=\alpha(\sx_i)$ and note that for every $\sr \in \cF_{k,s}(\R)$,
$$\alpha(\sr) = \sr(\alpha_1,\ldots,\alpha_k)$$
Let $\g$ be an $s$-step nilpotent real Lie algebra. The group $\Aut(\cF_{k,s})(\R)$ acts on $k$-tuples $\g^k$ as follows: $$\alpha \cdot (X_1,\ldots,X_k) =(\alpha_1(X_1,\ldots,X_k),\ldots,\alpha_k(X_1,\ldots,X_k)).$$
Note that this action is algebraic and preserves the measure class of the Lebesgue measure $\lambda$ on $\g^k$:
For $g\in\Aut(\cF_{k,s})$, the pushforward $g_*\lambda$ is absolutely continuous with respect to $\lambda$, with density given by the Radon-Nikodym cocycle $c(g,\bX)$.
Moreover, the Radon-Nikodym cocycle $c(g,\bX)$ is continuous in $(g,\bX) \in  \Aut(\cF_{k,s})(\R) \times \g^k$.
In fact $\Aut(\cF_{k,s})(\R)$ preserves the commutator ideal $[\cF_{k,s},\cF_{k,s}]$ and thus acts on the quotient $\cF_{k,s}/[\cF_{k,s},\cF_{k,s}]$, which is isomorphic to $\R^k$. This yields a natural epimorphism from the group $\Aut(\cF_{k,s})(\R)$ onto $\GL_k(\R)$ with unipotent kernel; the cocycle is independent of $\bX$ and given by the determinant of the image of $g$ under this epimorphism.

We are going to show:

\begin{proposition}[Ergodic action by automorphisms]\label{erg}
If $k \geq \dim \g/[\g,\g]$, then the action of $\Aut(\cF_{k,s})(\Q)$ on $\g^k$ is ergodic.
\end{proposition}

To prove this, we first recall:

\begin{lemma}\label{dense}
Let $G$ be a real Lie group, $H \leq G$ a closed subgroup and $\nu$ a quasi-invariant measure on $G/H$ with continuous Radon-Nikodym cocycle.
Then every dense subgroup of $G$ acts ergodically on $G/H$.
\end{lemma}
\begin{proof} We need to show that if $f \in L^{\infty}(G/H)$ is invariant under a dense subgroup of $G$, then it is constant almost everywhere.
Let $c:G \times G/H \to \R$ be the Radon-Nikodym cocycle, i.e. $\frac{d(g_*\nu)}{d\nu}(x)=c(g^{-1},x)$.
Then for every continuous and compactly supported function $\phi$ on $G/H$ and every sequence of elements $g_n \in G$ converging to $g \in G$, the sequence of functions $x \mapsto \phi(g_nx)c(g_n,x)$ is uniformly converging to $\phi(gx)c(g,x)$ on $G/H$. It follows that for every $f \in L^{\infty}(G/H)$, $\langle g_nf,\phi \rangle := \int_{G/H} f(g_n^{-1}x)\phi(x)d\nu(x)$ converges to $\langle gf,\phi \rangle$. Therefore, by duality, if $f\in L^{\infty}(G/H)$ is invariant under a dense subgroup of $G$, then for every $g\in G$  we have $gf=f$ almost everywhere.
The measurable set $\Omega=\{(g,x) \in G\times G/H\ |\ f(g^{-1}x) \neq f(x)\}$ is such that every vertical fiber $\Omega_g:=\{x \in G/H\ |\ f(g^{-1}x) \neq f(x)\}$ has $\nu$-measure zero.
By Fubini's theorem, $\Omega$ has measure zero, and $\nu$-almost every horizontal fiber $\Omega^x=\{g \in G \ |\  f(g^{-1}x) \neq f(x)\}$ has measure zero.
Fix $x$ such that $\Omega^x$ has zero measure.
The set $E=\{g^{-1}\cdot x\ ;\ g\in G\setminus\Omega^x\}$ has full measure in $G/H$ and $f$ is constant on $E$.
\end{proof}

\begin{proof}[Proof of Proposition \ref{erg}] The group $\Aut(\cF_{k,s})$ is an algebraic group defined over $\Q$. Hence the group of $\Q$-points $\Aut(\cF_{k,s})(\Q)$ is dense in the group of $\R$-points $\Aut(\cF_{k,s})(\R)$ (in this case this can be checked directly on the reductive part, which is $\GL_k$, and the unipotent part). Therefore it suffices to show that  $\Aut(\cF_{k,s})(\R)$ admits a Zariski open orbit on $\g^k$, when  $k \geq d=\dim(\g/[\g,\g])$. Indeed, its complement will have Lebesgue measure zero, while the open orbit will be a homogeneous space $G/H$
of $G:=\Aut(\cF_{k,s})(\R)$ with Lebesgue measure coming from $\g^k$ as quasi-invariant measure. Lemma \ref{dense} then implies that $\Aut(\cF_{k,s})(\Q)$ acts ergodically.

To see that $\Aut(\cF_{k,s})(\R)$ admits a Zariski open orbit when $k\geq d$, observe that any two $k$-tuples $\mathbf{X} , \mathbf{X}'$ of points in $\g^k$ with the property that their reductions modulo $[\g,\g]$ generate $\g/[\g,\g]$ as a vector space must be in the same orbit of $\Aut(\cF_{k,s})(\R)$.
Indeed, since $k\geq d=\dim (\g/[\g,\g])$, we can find an element of $\GL_k(\R)$ such that $g\mathbf{X}$ and $\mathbf{X}'$ have the same reduction modulo $[\g,\g]$. We can thus assume that $\mathbf{X}' - \mathbf{X}$ belongs to $[\g,\g]^k$.
Now, the fact that the tuple $\mathbf{X}=(X_1,\dots,X_k)$ generates $\g$ modulo $[\g,\g]$ implies that every element of $[\g,\g]$ can be written as $\alpha(X_1,\ldots,X_k)$ for some $\alpha \in [\cF_{k,s},\cF_{k,s}]$.
Therefore $\mathbf{X}' - \mathbf{X} = (\alpha_1(\mathbf{X}),\dots,\alpha_k(\mathbf{X}))$ for some $\alpha_i \in [\cF_{k,s},\cF_{k,s}]$. But every endomorphism of $\cF_{k,s}$ of the form $\sx_i \mapsto \sx_i + \alpha_i(\sx_1,\ldots,\sx_k)$ with $\alpha_i \in [\cF_{k,s},\cF_{k,s}]$ is invertible, hence belongs to $\Aut(\cF_{k,s})(\R)$ as desired.

To finish, simply note that when $k \geq d$, the set of $k$-tuples $\mathbf{X}$ of points in $\g^k$ such that their reduction modulo $[\g,\g]$ spans $\g/[\g,\g]$ is a non-empty Zariski open subset of $\g^k$.
\end{proof}

\subsection{Critical exponent of a nilpotent Lie group}

We will  prove below Theorem \ref{zerone}. It will be a consequence of Proposition~\ref{erg} and Proposition~\ref{invariance} below. In the statement of this proposition, we identify the connected and  simply connected nilpotent Lie group $G$ with its Lie algebra $\g$ via the exponential map $\exp: \g \to G$, which is a diffeomorphism. This allows to view  $\Aut(\cF_{k,s})(\R)$ as acting on $G^k$ rather than $\g^k$. 

\begin{proposition}\label{invariance} Given $\beta\geq 0$, the set of $\beta$-diophantine $k$-tuples in $G^k$ is Lebesgue measurable and invariant under the action of $\Aut(\cF_{k,s})(\Q)$.
\end{proposition}

Recall that two groups $\Gamma$ and $\Gamma'$ are \emph{commensurable} if their intersection has finite-index in both $\Gamma$ and $\Gamma'$.
We divide the proof of Proposition~\ref{invariance} into two lemmas.

\begin{lemma}\label{betacom}
Let $G$ be a simply connected nilpotent Lie group equipped with a left-invariant distance $d(\cdot,\cdot)$, and $\Gamma$ a finitely generated subgroup of $G$.
If $\Gamma$ is $\beta$-diophantine, then any subgroup of $G$ commensurable to $\Gamma$ is also $\beta$-diophantine.
\end{lemma}
\begin{proof}
Let $\Gamma$ be $\beta$-diophantine and $\Gamma'$ commensurable to $\Gamma$.
Then $\Gamma\cap\Gamma'$ has finite-index in $\Gamma'$ and therefore, there exists a normal subgroup $\Gamma_0<\Gamma\cap\Gamma'$ that has finite-index in $\Gamma'$.
In particular, for some integer $p$, any element $\gamma\in\Gamma'$ has $\gamma^p\in\Gamma_0$.
Moreover, $\Gamma_0$ is included in $\Gamma$, so it is $\beta$-diophantine.
Let $S$ and $S_0$ be symmetric generating sets for $\Gamma'$ and $\Gamma_0$, respectively.
Since $\Gamma_0$ has finite index in $\Gamma'$, there exists a constant $C$ such that for all integer $n\geq 1$, $\Gamma_0\cap S^n\subset S_0^{Cn}$.
Now suppose $\gamma\in\Gamma'$ is an element of $S^n\backslash\{1\}$.
Then $\gamma^p$ is an element of $\Gamma_0\cap S^{pn}$ and therefore
$$\gamma^p \in S_0^{Cpn}.$$
Using that $|S_0^n|$ grows polynomially and the fact that $\Gamma_0$ is $\beta$-diophantine, this implies
$$d(\gamma^p, 1) \gg |S_0^n|^{-\beta},$$
where $\gg$ means $\geq$ up to a positive multiplicative constant, and in turn,
$$d(\gamma,1) \geq \frac{1}{p}d(\gamma^p,1) \gg |S_0^n|^{-\beta}
\gg |S^n|^{-\beta}.$$
\end{proof}

Our second lemma is as follows.

\begin{lemma}\label{commens}
Let $\mathbf{X}:=(X_1,\ldots,X_k) \in \g^k$, let $\alpha \in \Aut(\cF_{k,s})(\Q)$ and set $\mathbf{X}':=\alpha(\mathbf{X}) \in \g^k$.
Then the subgroup of $G$ generated by $e^{X_1},\ldots,e^{X_k}$ is commensurable to the subgroup generated by $e^{X'_1},\ldots,e^{X'_k}$.
\end{lemma}
\begin{proof}
Recall that if $\gamma_1,\ldots,\gamma_k$ generate a nilpotent group, then for every integer $n \geq 1$, the subgroup generated by the powers $\gamma_1^n,\ldots,\gamma_k^n$ has finite index \cite[Lemma 4.4.]{raghunathan}.
By assumption, there is an integer $N$ such that each $NX'_i$ belongs to $\cF_{k,s}(\Z)(X_1,\ldots,X_k)$, that is $NX'_i$ is an integer linear combination of commutators in $X_1,\ldots,X_k$. However recall \cite[Lemma 3.5.]{abrs} that there is an integer $M$ such that $e^{M\sr(X_1,\ldots,X_k)}$ belongs to the subgroup generated by $e^{X_1},\ldots,e^{X_k}$ for any $\sr \in \cF_{k,s}(\Z)$. It follows that each $(e^{X'_i})^{MN}$ belongs to the subgroup generated by $e^{X_1},\ldots,e^{X_k}$. Interchanging $\mathbf{X}$ and $\mathbf{X}'$, the lemma follows.
\end{proof}

\begin{proof}[Proof of Proposition~\ref{invariance}]
It is clear from the definition that the subset of elements $\bg$ such that $\Gamma_{\bg}$ is $\beta$-diophantine is a Borel subset. Now suppose that $\mathbf{g}=(g_1,\ldots,g_k)$ is such that $\Gamma_\mathbf{g}$ is  $\beta$-diophantine and let $\mathbf{g}'$ be in the $\Aut(\cF_{k,s})(\Q)$-orbit of $\mathbf{g}$.
This means that $\log(\mathbf{g})$ and $\log(\mathbf{g}')$ are in the same $\Aut(\cF_{k,s})(\Q)$-orbit. So Lemma \ref{commens} implies that  $\Gamma_\mathbf{g}$ and $\Gamma_\mathbf{g'}$ are commensurable.
Since $\Gamma_{\mathbf{g}}$ is $\beta$-diophantine, it follows from Lemma~\ref{betacom} that $\Gamma_{\mathbf{g}'}$  also is $\beta$-diophantine.
\end{proof}

\begin{proof}[Proof of Theorem \ref{zerone}]
In view of Proposition \ref{invariance}, the set $D_\beta$ of $k$-tuples such that $\Gamma_{\bg}$ is $\beta$-diophantine is measurable and invariant under the action of $\Aut(\cF_{k,s})(\Q)$. Since this action is ergodic by Proposition \ref{erg}, we conclude that $D_\beta$ is either null or conull. If $\widehat{\beta_k}$ denotes the infimum of all $\beta\geq 0$ for which $D_\beta$ is conull, then $D_\beta$ will be conull when $\beta>\widehat{\beta_k}$ and null if $\beta<\widehat{\beta_k}$.
\end{proof}

\begin{remark}The property of being diophantine for a $k$-tuple $\mathbf{g}$ does not depend only on the projection of $\mathbf{g}$ on the abelianization of $G$. It is easy to construct examples showing that one may have two tuples $\mathbf{g}$ and $\mathbf{g}'$ such that $\mathbf{g}=\mathbf{g}'$ modulo $[G,G]$, but $\Gamma_{\mathbf{g}}$ is diophantine, while $\Gamma_{\mathbf{g}'}$ is not diophantine in $G$ and $\Gamma_{\pi(\mathbf{g}')}$ is not diophantine in $G/[G,G]$ either, where $\pi$ is the projection modulo $[G,G]$. In particular, $\Aut(\cF_{k,s})(\Q)$ in Proposition \ref{invariance} cannot be replaced by the subgroup $\GL_k(\Q)$.
\end{remark}

\begin{remark} In Theorem \ref{zerone}, the distance $d(\cdot,\cdot)$ is only assumed to be left-invariant, and need not be geodesic. The alternative argument given in Section~\ref{section:critical}, which is based on quantitative non-divergence, requires the distance to be  geodesic (and hence a Carnot-Carathéodory-Finsler metric by \cite{berestowski}).
\end{remark}

\section{Critical exponent for the Heisenberg group}\label{section:heisenberg}

As a warm-up, we now present an explicit computation of the critical exponent of the 3-dimensional Heisenberg group. The method is elementary, using the Borel-Cantelli lemma combined with an ad-hoc Remez-type inequality for a certain family of quadratic forms (cf. Lemma \ref{levelquadratic} below). The relationship between this elementary method and the Kleinbock-Margulis type approach developed later in this paper will be discussed at the end of this section.

Here $G$ will always denote the $3$-dimensional Heisenberg group, consisting of $3\times 3$ unipotent upper-triangular matrices.
It will be convenient for us to view $G$ as the space $\R^3$, endowed with the group law
\[ (x,y,z)*(x',y',z') = (x+x',y+y',z+z'+xy').\]
Recall the definition of the diophantine exponent of a finitely generated subgroup of $G$, made at the beginning of Section \ref{section:ergodicity}.
We want to prove the following.

\begin{thm}[Critical exponent for the Heisenberg group]
\label{criticalheisenberg}
Let $k\geq 2$. Then for almost every $k$-tuple $\bg=(g_1,\ldots,g_k)$ in $G$,
\[ \beta(\Gamma_\bg) = 1-\frac{1}{k}-\frac{2}{k^2}.\]
\end{thm}

The diophantine exponent here is computed with respect to any left-invariant Riemannian metric on $G$ (equivalently for $\max\{|x|,|y|,|z|\}$). For the proof of Theorem~\ref{criticalheisenberg}, we will need a few definitions.
Recall that if $w$ is a word on $k$ letters, i.e. an element of the free group $F_k$ over $k$ generators $x_1,\dots,x_k$, it induces a \emph{word map}
\[ \begin{array}{cccc}
w_G: & G^k & \to & G\\
& (g_1,\dots,g_k) & \mapsto & w(g_1,\dots,g_k),
\end{array} \]
where $w(g_1,\dots,g_k)$ is the element of $G$ obtained by substituting each letter $x_i$ by the element $g_i$.
The group $F_{k,G}$ of word maps in $k$ letters on $G$ is defined to be the set of all such maps, with product law given by $w_Gw_G'=(ww')_G$.
The length $\ell(\omega)$ of a word map $\omega$ in $F_{k,G}$ is the minimal length of a word $w$ such that $\omega=w_G$.

In the case of the Heisenberg group, one can describe the group of word maps very explicitly.
For any two elements $g$ and $h$ in $G$, $[g,h]$ denotes the commutator of $g$ and $h$, defined by $[g,h] = ghg^{-1}h^{-1}.$ The subgroup $[G,G]$ generated by commutators coincides with the center $Z$ of $G$ and is the set of elements $(0,0,z)$, for $z\in\R$.

\begin{prop}[Word maps on the Heisenberg group]
\label{explicitwordmaps}
Let $k\geq 2$. For each word map $\omega$ on $k$ letters on $G$, there exist integers $n_i$, $1\leq i\leq k$ and $n_{ij}$, $1\leq i<j\leq k$ such that for all $\bg=(g_1,\dots,g_k)$ in $G^k$,
\[ \omega(\bg) = g_1^{n_1}\dots g_k^{n_k} \prod_{1\leq i<j\leq k} [g_i,g_j]^{n_{ij}}.\]
Moreover, the $n_i$ and $n_{ij}$ are uniquely determined by $\omega$ and there exists a constant $C>0$ depending only on $k$ such that
\[ \frac{1}{C}\ell(\omega)
\leq \max_{l,i,j}\{|n_l|,|n_{ij}|^{\frac{1}{2}}\}
\leq C\ell(\omega).\]
\end{prop}
\begin{proof}
The existence of the $n_i$ and $n_{ij}$ is proved by elementary operations on a word representing $\omega$, using that all commutators lie in the center of $G$, because $G$ is $2$-step nilpotent.
To verify uniqueness, it suffices to show that no non-trivial family of integers $n_l,n_{ij}$ yields the trivial word map; this can be checked directly, by expliciting the word maps in the $(x,y,z)$ coordinates for $G$.
The statement about the length of $\omega$ can also be proved directly, using the identity in $G$, $[g_i^n,g_j^m]=[g_i,g_j]^{nm}$.
We leave the details to the reader.
\end{proof}

We are now ready to prove the following theorem, which will easily imply Theorem~\ref{criticalheisenberg}.

\begin{thm}
\label{heisenbergwords}
Fix an integer $k\geq 3$.
\begin{itemize}
\item If $\alpha > k^2-k-2$, then for almost every $k$-tuple $\bg\in G^k$, $d(\omega(\bg),1) \geq \ell(\omega)^{-\alpha}$ for all but finitely many $\omega\in F_{k,G}$.
\item If $\alpha<k^2-k-2$, then for all $\bg\in G^k$, there are infinitely many $\omega\in F_{k,G}$ for which $d(\omega(\bg),1)<\ell(\omega)^{-\alpha}$.
\end{itemize}
\end{thm}

We decompose the proof into two lemmas, studying first the word maps $\omega$ on $G$ that are non-trivial modulo the center of $G$, i.e. those for which some $n_i$ is non-zero.

\begin{lem}[Diophantine property outside the derived subgroup]
\label{abelianized}
Let $k\geq 2$.
If $\alpha>\frac{k}{2}-1$, then for almost every $k$-tuple $\bg$, we have $d(\omega(\bg),1)\geq\ell(\omega)^{-\alpha}$ for all but finitely many word maps $\omega$ that are non-trivial modulo the center of $G$.
\end{lem}
\begin{proof}
Suppose $\bg=(g_1,\dots,g_k)$ is chosen at random in $G^k$ according to the Haar measure on a compact subset.
Then, the projection $\bar{\bg}=(\bar{g}_1,\dots,\bar{g}_k)$ to $(G/Z)^k$ is a random $k$-tuple in $(G/Z)^k\simeq (\R^2)^k$, and its law is absolutely continuous with respect to the Lebesgue measure.
By a standard application of the Borel-Cantelli lemma, we know that almost surely, if $\alpha>\frac{k}{2}-1$, then
\[ \| n_1\bar{g}_1+\dots+n_k\bar{g}_k\| \geq (\max|n_i|)^{-\alpha},\]
for all but finitely many $(n_1,\dots,n_k)$ in $\Z^k$.
This proves the lemma.
\end{proof}

We now need to study word maps that are trivial modulo the center.

\begin{lem}[Diophantine property inside the derived subgroup]
\label{derived}
Let $k\geq 2$.
If $\alpha>k^2-k-2$, then for almost every $k$-tuple $\bg$, we have $d(\omega(\bg),1)\geq\ell(\omega)^{-\alpha}$ for all but finitely word maps $\omega$ of the form $\omega:\bg\mapsto \prod_{1\leq i<j\leq k}[g_i,g_j]^{n_{ij}}$.
\end{lem}
\begin{proof}
Let $\omega:\bg\mapsto\prod_{1\leq i<j\leq k}[g_i,g_j]^{n_{ij}}$.
In the $(x,y,z)$ coordinates, the map $\omega$ is given by
\[ \begin{array}{cccc}
\omega: & G^k & \to & G\\
& (x_i,y_i,z_i)_{1\leq i\leq k} & \mapsto & (0,0,\sum_{1\leq i<j\leq k} n_{ij}(x_iy_j-y_ix_j))
\end{array}.\]
If $\mu$ denotes the Lebesgue measure on the ball of radius $1$ centered at $0$ in $G^k$, Lemma~\ref{levelquadratic} below implies that
\[ \mu(\{\bg\,|\, d(\omega(\bg),1)\leq\ell(\omega)^{-\alpha}\})
\leq \frac{2^{2k+1}}{\ell(\omega)^\alpha\max|n_{ij}|} (1+\log(\ell(\omega)^\alpha\sqrt{2k}\max |n_{ij}|)).\]
However, by Lemma~\ref{explicitwordmaps}, $\ell(\omega)\simeq (\max|n_{ij}|)^{\frac{1}{2}}=\|n\|^{\frac{1}{2}}$, and therefore,
\[ \mu(\{\bg\,|\, d(\omega(\bg),1)\leq\ell(\omega)^{-\alpha}\})
\ll \frac{\log\|n\|}{\|n\|^{1+\frac{\alpha}{2}}}.\]
Since every word map trivial modulo the derived group corresponds to a unique $\frac{k(k-1)}{2}$-tuple $(n_{ij})$, we find
\[ \sum_{\omega\equiv 0\ \textrm{mod}\ Z}
\mu(\{\bg\,|\, d(\omega(\bg),1)\leq\ell(\omega)^{-\alpha}\})
 \ll \sum_{(n_{ij})\in\Z^{\frac{k(k-1)}{2}}-\{0\}} \frac{\log\|n\|}{\|n\|^{1+\frac{\alpha}{2}}}
 < \infty,\]
because $1+\frac{\alpha}{2}>\frac{k(k-1)}{2}$.
The lemma then follows from the Borel-Cantelli lemma.
\end{proof}

We now state and prove the elementary lemma used in the above proof.

\begin{lem}
\label{levelquadratic}
Let $q$ be a quadratic form on $\mathbb{R}^d$, and assume that in some orthogonal basis, $q$ can be expressed without squares as
\[q(x)=\sum_{k<l}a_{kl}x_kx_l.\]
If $\mu$ denotes the Lebesgue measure on the ball of radius $1$ centered at $0$ in $\R^d$, then for all $\eps>0$
\[ \mu(\{x\in\R^d \,|\, |q(x)|\leq\eps\}) \leq \frac{2^{d+1}\eps}{\max |a_{kl}|}\left(1+\log^+\pa{\frac{\sqrt{d}\max|a_{kl}|}{\epsilon}}\right),\]
where $\log^+x=\max(0,\log x)$.
\end{lem}
\begin{proof}
Choose $k_0$ and $l_0$ such that $|a_{k_0l_0}|=\max_{k,l}|a_{kl}|$. As we may permute the indices, we may assume without loss of generality that $(k_0,l_0)=(1,2)$. Then write
\[ q(x)=x_1\cdot(\sum_{l\geq2} a_{1l}x_l)+y(x_2,\dots,x_d)=x_1\langle x',a\rangle +y(x'),\]
where $x'=(x_2,\dots,x_d)$, $a=(a_{12},\dots,a_{1d})$, and $\langle,\rangle $ denotes the usual inner product in $\mathbb{R}^{d-1}$.
Let $\lambda$ denote the Lebesgue measure on the real line.
For fixed $(x_2,\dots,x_d)$, observe that
\begin{align*}
\lambda(\{x_1\in [-1,1] \ |\ |x_1\langle x',a\rangle -y(x')|\leq\epsilon\})
& \leq 2\min(1,\frac{\epsilon}{|\langle x',a\rangle |}).
\end{align*}
Then write
\begin{align*}
\mu (\{ x\in\mathbb{R}^d \ |\  q(x)\leq \epsilon\}) & \leq \int_{B_{\mathbb{R}^{d-1}}(0,1)} \lambda(\{x_1\in[-1,1];|x_1\langle x',a\rangle -y(x')|\leq\epsilon\})\,\mathrm{d}x'\\
& \leq 2\int_{B_{\mathbb{R}^{d-1}}(0,1)} \min(1,\frac{\epsilon}{|\langle a,x'\rangle |})\,\mathrm{d}x'\\
& \leq 2^d \int_{-1}^1\min\left(1,\frac{\epsilon}{\|a\|t}\right)\,\mathrm{d}t\\
& = \frac{2^{d+1}\epsilon}{\|a\|}\left(1+\log^+ \frac{\|a\|}{\epsilon}\right).
\end{align*}
Since $\max|a_{kl}| \leq \|a\| \leq \sqrt{d}\max|a_{kl}|$, we find
\[
\mu (\{ x\in\mathbb{R}^d \ |\  q(x)\leq \epsilon\})
\leq \frac{2^{d+1}\epsilon}{\max |a_{kl}|}\left(1+\log^+\frac{\sqrt{d}\max|a_{kl}|}{\epsilon}\right).
\]
\end{proof}

We can now conclude the proofs of Theorems~\ref{heisenbergwords} and \ref{criticalheisenberg}.

\begin{proof}[Proof of Theorem~\ref{heisenbergwords}]
The first assertion follows from combining Lemmas~\ref{abelianized} and \ref{derived}, noting also that for $k\geq 2$, one has $\frac{k}{2}-1\leq k^2-k-2$.

For the second assertion, we note we have $k(k-1)/2$ commutators $[g_i,g_j]$ lying in $Z\simeq \mathbb{R}$. Therefore, given a positive integer $n$, Dirichlet's pigeonhole argument shows that there exist integers $n_{ij}$, $1\leq i<j\leq k$ such that $|n_{ij}|\leq n^2$ and
\[ \left|\sum_{i<j} n_{ij}[g_i,g_j]\right| \leq Cn^{-k^2+k+2},\]
where $C$ is a constant depending only on $\bg$.
Thus, for the word map $\omega:\bg\mapsto\prod_{i<j}[g_i,g_j]^{n_{ij}}$, we get
\[ d(\omega(\bg),1) \leq C \ell(\omega)^{-k^2+k+2}.\]
\end{proof}

\begin{proof}[Proof of Theorem~\ref{criticalheisenberg}]
By Proposition~\ref{explicitwordmaps}, the number of word maps of length at most $n$ is bounded above and below by a positive constant times $n^{k+2\frac{k(k-1)}{2}}=n^{k^2}$.
However, we know from \cite[Lemma~2.5]{abrs} that, for almost every $k$-tuple $\bg$, the group $\Gamma_{\bg}$ is isomorphic to $F_{k,G}$, so that, if $V_{\bg}(n)$ denotes the number of elements in the ball of radius $n$ with respect to the generating set $\bg$, there exist positive constants $c_1, c_2$ such that
\[ c_1 n^{k^2} \leq V_{\bg}(n) \leq c_2 n^{k^2}.\]
Together with Theorem~\ref{heisenbergwords}, this shows that for almost every $k$-tuple $\bg$,
\[ \beta(\Gamma_\bg) = \frac{k^2-k-2}{k^2} = 1-\frac{1}{k}-\frac{2}{k^2}.\]
\end{proof}

\begin{remark}
Let
\[ \begin{array}{cccc}
\varphi: & G^k & \to & \R^{\frac{k(k-1)}{2}}\\
& (x_i,y_i,z_i) & \mapsto & (x_iy_j-y_ix_j)_{1\leq i<j\leq k}
\end{array}.\]
Lemma~\ref{derived} is equivalent to saying that the pushforward under $\varphi$ of the Haar measure on (a ball of) $G^k$ is extremal. So an alternative proof consists in using the Kleinbock-Margulis theorem \cite[Theorem A]{kleinbock-margulis}: all one is left to check is that the image of $\varphi$ is not contained in a hyperplane.

In the rest of this paper, this will be our approach to compute the critical exponent of an arbitrary connected simply connected rational nilpotent Lie group.  Clearly there is little hope to use a direct elementary approach in the spirit of Lemma~\ref{derived} to handle general nilpotent groups. 
We will reduce the problem to studying the extremality of certain maps from  $G^k$ to $G$. If the target space of these maps were one-dimensional, the Kleinbock-Margulis theory would in general be enough. However this is typically not the case and one is thus naturally led to develop a suitable matrix analogue of their theory, which is what we do in the next two sections. The translation to a problem about (weighted) diophantine approximation on subanifolds of matrices is expounded in Section~\ref{section:critical}.
\end{remark}

\section{Quasi-norms, Schubert cells, and pencils}

The main goal of the next three sections is to establish Theorems~\ref{rational-exponent-babythm} and \ref{general-inequality} from the introduction, which address a problem raised by Kleinbock and Margulis in \cite[6.2]{kleinbock-margulis} and studied in \cite{kleinbock-margulis-wang} and \cite{beresnevich-kleinbock-margulis}.
The proof will be split into two parts, which correspond to Sections~\ref{quasinorm-sec} and \ref{section:submodular}, respectively. First, the upper bound and the fact that the almost sure exponent depends only on the Zariski closure will be proved as Theorem~\ref{uppboundnu} below, which is a consequence of the quantitative non-divergence estimates on the space of lattices, and the proof will be given in the slightly more general setup of locally good measures.
Second, we shall show that equality holds when the Zariski closure is defined over $\Q$. This will be Theorem~\ref{rathm}, and will be obtained as a consequence of the \emph{submodularity lemma} proved in Section~\ref{section:submodular}. 
In the present section, we describe the geometric objects involved in our diophantine problem, and explain what Dirichlet's pigeonhole argument, giving the lower bound on the exponent, becomes in this general setting.

As before, $V$ and $E$ are two finite-dimensional real vector spaces, and $\Delta$ is a lattice in $V$.
Given a measure $\nu$ (say a probability measure) on $\Hom(V,E)$, we want to study the diophantine properties with respect to $\Delta$ of a random point $x$ in $\Hom(V,E)$ chosen according to the distribution $\nu$.
For the application to nilpotent groups that we develop in Section~\ref{section:critical}, we need to compute the diophantine exponent of an element $x$ in $\Hom(V,E)$, which is defined in terms of certain quasi-norms on $V$ and $E$.

\begin{defn}\label{quasinorm}
A \emph{quasi-norm} (resp. \emph{local quasi-norm}) on $V$ is a map $V \to \R_+$, $v \mapsto |v|$, for which there exist $C>0$, a basis $u_1,\ldots,u_d$ of $V$ and positive real numbers $\alpha_1,\ldots,\alpha_d$ such that
\begin{equation}\label{vdef}
\frac{1}{C}|v| \leq \max_{1\leq i\leq d} |u_i^*(v)|^{1/\alpha_i} \leq C |v|
\end{equation}
outside a neighborhood of the origin (resp. in a neighborhood of the origin), where $u_1^*,\ldots,u_d^*$ is the dual basis.
\end{defn}

Given a quasi-norm $|\cdot|$ on $V$ and a local quasi-norm $|\cdot|'$ on $E$, we define the \emph{diophantine exponent} of $x$ in $\Hom(V,E)$ by
\begin{equation}\label{betdef} \beta(x)=\inf\{\beta>0\ |\ \exists c>0:\, \forall v\in\Delta\setminus\{0\},\,
|x(v)|'\geq c|v|^{-\beta}\}.\end{equation}

\begin{remark}
At a first reading it is fair to assume that $|\cdot|$ and $|\cdot|'$ are chosen to be genuine norms on $V$ and $E$.
\end{remark}

\subsection{Quasi-norms and local quasi-norms}\label{subsecquasi}

For later use, we now record some elementary properties of quasi-norms and local quasi-norms.
We say that two quasi-norms (resp. local quasi-norms) are \emph{comparable} (or \emph{equivalent}) if their ratio is bounded and bounded away from zero outside a neighborhood of the origin (resp. in a neighborhood of the origin).
Let $V$ and $E$ be real vector spaces of respective dimension $d$ and $e$, and fix $|\cdot|$ a quasi-norm on $V$ and $|\cdot|'$ a local quasi-norm on $E$.
As in Definition~\ref{quasinorm}, we fix bases $(u_i)_{1\leq i\leq d}$ and $(u_i')_{1\leq i\leq e}$ for $V$ and $E$, together with positive real numbers $\alpha_1,\dots,\alpha_d$ and $\alpha_1',\dots,\alpha_e'$ such that

\[ \frac{1}{C}|v|
\leq \max_{1\leq i\leq d} |u_i^*(v)|^{1/\alpha_i} \leq C |v|\]
outside a neighborhood of $0$ in $V$, and
\[ \frac{1}{C}|v|'
\leq \max_{1\leq i\leq d} |u_i'^*(v)|^{1/{\alpha_i'}} \leq C |v|'\]
in a neighborhood of $0$ in $E$.


\bigskip

\noindent {\bf 1.} If we reorder the vectors $u_i$ so that $\alpha_1\geq \ldots\geq \alpha_d>0$ and let $V_i=\langle u_{1},\ldots,u_{i-1}\rangle$, we obtain a full flag $0=V_1<V_2<\ldots<V_d<V$, and  $|v|$ is comparable to
\begin{equation}\label{defquasi}
\max_{1\leq i \leq d} d(v,V_i)^{\frac{1}{\alpha_i}},
\end{equation}
where $d(v,V_i)$ is the distance to $V_i$ in some fixed Euclidean structure on $V$.
Similarly, if $\alpha_1'\geq\dots\geq\alpha_e'>0$, then $|v|'$ is comparable to
\begin{equation}\label{defquasilocal}
\max_{1\leq i \leq e} d(v,V_i')^{\frac{1}{\alpha_i'}},
\end{equation}
where $V_i'=\langle u_{i+1}',\dots,u_e'\rangle$.
For the rest of the paper, we assume that $\alpha_1\geq\dots\geq\alpha_d$ and $\alpha_1'\geq\dots\geq\alpha_e'$.
\bigskip

\noindent {\bf 2.} (subspace) If $W\leq V$ is a subspace, then the restriction of $|\cdot|$ to $W$ is a quasi-norm, associated to the flag $0=V_1 \cap W \leq \ldots \leq V_d \cap W \leq W$ in the sense that $|w|$ is comparable to
$$\max_{1\leq i \leq d} d(w,V_i \cap W)^{1/\alpha_i},$$
when $w \in W$. This is easily seen, since $d(w,W\cap V_i)$ is comparable to $d(w,V_i)$  when $w$ varies in $W$.
Similarly, the restriction of $|\cdot|'$ to a subspace $F$ in $E$ is a local quasi-norm, associated to the flag $0=V_e' \cap F \leq \ldots \leq V_1' \cap F \leq F$.  
\bigskip

\noindent {\bf 3.} (quotient) If $W \leq V$ is a subspace, then $|\cdot|$ induces a quasi-norm on $V/W$ by setting, for $\overline{v} := v\mod W$,
$$|\overline{v}| := \inf_{w \in W} |v+w|.$$
It is easy to see that this indeed defines a quasi-norm $V/W$.
It is associated to the flag $\{V_i/(V_i \cap W)\}$ of $V/W$, in the sense that it is comparable to
$$\max_{1\leq i \leq d} d(\overline{v},V_i/(V_i \cap W))^{1/\alpha_i},$$
where $V_i/(V_i \cap W) \leq V/W$.
To see this, note that $W$ has an adapted complement, namely a subspace $W'$ such that $(W \cap V_i) \oplus (W' \cap V_i) = V_i$ for every $i=1,\ldots,d$, and  that $|\overline{v}|$ is comparable to the restriction of $|v|$ to $W'$.
\bigskip

\noindent {\bf 4.} (triangle inequality) There is $C>0$ such that
$$|v+w| \leq C(|v|+|w|)$$
for every $v,w$ in $V$ outside a neighborhood of the origin in $V$, and
\[ |v+w|' \leq C(|v|'+|w|')\]
for every $v,w$ inside a neighborhood of the origin in $E$.

\bigskip

\noindent {\bf 5.} (volume of large balls)
From the quasi-norm $|\cdot|$ on $V$, we define a function $\psi: \Grass(V) \to \R_+$ by
$$\psi(W)=\sum_{i \in I(W)} \alpha_i,$$
where
$$I(W)=\{i\in \{1,\ldots,d\} \ |\ \dim(V_{i+1} \cap W) = \dim (V_i \cap W)+1\},$$
with the convention that $V_{d+1}=V$ and $\alpha_{d+1}=0$. Clearly $\psi$ is non-decreasing because $I$ is. Moreover $-\psi$ is submodular (cf. Lemma \ref{submodular}). This is easily seen, since reorganizing the sum yields
\begin{equation}\label{psidec}\psi(W)=\sum_{i=1}^d (\alpha_i-\alpha_{i+1}) \dim (V_{i+1} \cap W),\end{equation} and $\alpha_i \geq \alpha_{i+1}$. Moreover, $\psi$ determines the volume of the ball of radius $Q$ restricted to $W$: for all  $Q>1$,
$$\vol\{ w \in W \ |\ |w| \leq Q \} \simeq Q^{\psi(W)}.$$
Here and hereafter, we write $x\simeq y$ if there exist positive constants $c, C>0$ such that $cx\leq y\leq Cx$.
The constants $c$ and $C$ are allowed to depend on $W$ and $|\cdot|$, but \emph{not} on $Q$.
Note also that if $W'$ is an adapted complement to $W$ in $V$, then
$$I(W')=\{1,\ldots,d\}\setminus I(W) = \{i \in \{1,\ldots,d\} \ |\ \dim (\frac{V_{i+1}}{V_{i+1}\cap W}) =  \dim (\frac{V_{i}}{V_{i}\cap W})+1\},$$
and hence $\psi(W')=\psi(V)-\psi(W)$ and
$$\vol\{ \overline{v} \in V/W \ |\ |\overline{v}| \leq Q \} \simeq Q^{\psi(V)-\psi(W)}.$$

\bigskip

\noindent {\bf 6.} (volume of small balls)
We also define a function $\phi: \Grass(E) \to \R_+$ by
$$\phi(F)=\sum_{i \in J(F)} \alpha_i',$$
where $$J(F)=\{i\in \{1,\ldots,e\} \ |\ \dim(V_{i}' \cap F) = \dim (V_{i-1}' \cap F)-1\},$$
with the convention that $V_{0}'=E$. Then $\phi$ is non-decreasing (because $J$ is) and submodular (cf. Lemma \ref{submodular}). This is easily seen, since $\alpha_i' \geq \alpha_{i+1}'$, and reorganizing the sum yields
\begin{equation}\label{phidec}
\phi(F)=\sum_{i=1}^e (\alpha_i'-\alpha_{i+1}')(\dim F- \dim (V_i' \cap F)).
\end{equation}
Moreover, $\phi$ determines the volume of the ball of radius $\eps$ restricted to $F$: within multiplicative constants (depending on $F$ and $|\cdot|'$ only), for all  $\eps \in (0,1)$,
$$\vol\{ w \in F \ |\ |w|' \leq \eps \} \simeq \eps^{\phi(F)}.$$

\subsection{Schubert varieties}\label{schubert}
Schubert varieties are certain distinguished closed algebraic subsets of the Grassmannian. See \cite{griffiths-harris} for the complex case and \cite{bochnak-coste-roy} for the real case. In this subsection we briefly recall their definition, because the pencils defined in the introduction give rise to Schubert varieties via the map $x \mapsto \ker x$ and because they will appear in the definition of locally good measures below.

As before $V$ is a $d$-dimensional real vector space and $0=V_1<V_2<\ldots<V_d<V_{d+1}=V$ is a full flag of subspaces. Let $n$ be an integer with $1\leq n\leq d$ and let $\sigma=(\sigma_1,\ldots,\sigma_n)$ be a sequence of integers such that $1 \leq \sigma_1 < \sigma_2 < \ldots < \sigma_n \leq d$. Let $\Grass_n(V)$ be the Grassmannian of $n$-dimensional subspaces of $V$. Recall that 
 $\Grass_n(V)$ can be realized as a closed (affine) algebraic subset of $M_{d,d}(\R)$ by assigning to each subspace the orthogonal projection onto it. The \emph{Schubert cell} of type $\sigma$ associated to the full flag $\{V_i\}_i$ is the subset $e(\sigma)$ of  all subspaces $W$ in $\Grass_n(V)$  such that $\dim(V_{\sigma_{i}+1} \cap W)=\dim(V_{\sigma_i} \cap W)+1$ for each $i=1,\ldots,n$, or equivalently in the notation of the previous subsection, such that $I(W)=\{\sigma_1,\ldots,\sigma_n\}$.

The cell $e(\sigma)$ is a subset of  $\Grass_n(V)$, and it is easy to see that its closure is a union of other cells, namely:
$$\overline{e(\sigma)} = \bigcup_{\tau \leq \sigma} e(\tau),$$
where we define the partial order $\tau \leq \sigma$ by requiring that $\tau_i \leq \sigma_i$ for each $i=1,\ldots,n$. This closure $\overline{e(\sigma)} $ is called the \emph{Schubert variety} of type $\sigma$.

It is worth mentioning that subsets of the form $\{W \in \Grass_n(V) ; \dim(W \cap V_j) \geq i\}$ for some indices $i,j$ are Schubert subvarieties and that the Schubert subvarieties associated to the full flag $\{V_i\}_i$ are precisely the intersections of subsets of this form.

Note that $ \dim(W \cap V_j) \geq i$ if and only if the family of vectors obtained by projecting a fixed basis of $V_j$ onto the orthogonal complement of $W$  has rank $< j-i$. In particular this subset of the Grassmannian is a closed algebraic subset and hence so are all Schubert subvarieties. Moreover the Schubert cell $e(\sigma)$ is Zariski open (and dense) in the Schubert subvariety $\overline{e(\sigma)}$.

\subsection{Dirichlet's principle}\label{dirich}

As before, $V$ and $E$ are real vector spaces of respective dimensions $d$ and $e$, and $\Delta$ is a lattice in $V$.
We fix a quasi-norm $|\cdot|$ on $V$ associated to a full flag $0=V_1<V_2<\ldots<V_d<V$ and a $d$-tuple $\alpha_1\geq\dots\geq\alpha_d>0$ via (\ref{defquasi}), and a local quasi-norm $|\cdot|'$ on $E$ defined by (\ref{defquasilocal}) using a decreasing full flag $E>V_1'>V_2'>\dots V_e'=0$ and an $e$-tuple $\alpha'_1\geq\dots\geq\alpha'_e>0$.
To this data are associated the volume exponents functions $\psi:\Grass(V)\to\R_+$ and $\phi:\Grass(E)\to\R_+$ defined in points {\bf 5.} and {\bf 6.} of the Subsection~\ref{subsecquasi}.

Recall that the diophantine exponent of a homomorphism $x\in \Hom(V,E)$ with respect to these quasi-norms is defined by
\begin{equation}
\label{diophexpdef} \beta(x) 
= \inf\{\beta>0 \ |\ \exists c>0:\ \forall v\in\Delta\setminus\{0\},\
|xv|' \geq c|v|^{-\beta}\}.
\end{equation}

A subspace $W\leq V$ is called $\Delta$-rational if $W \cap \Delta$ is a lattice in $W$.

\begin{proposition}[Dirichlet's principle]
\label{dirichlet}
For every $x\in \Hom(V,E)$ and every $\Delta$-rational subspace $W \leq V$,
\[ \beta(x) \geq  \frac{\psi(W\cap \ker x)}{\phi(xW)}.\]
\end{proposition}
\begin{proof} This is a version of the classical Dirichlet argument using the pigeonhole principle.
If $Q$ is a large parameter, the ball of radius $Q$ in $W$ for the quasi-norm $|\cdot|$ contains roughly (up to multiplicative constants) $Q^{\psi(W)}$ points of $\Delta$, which are mapped to a ball in $xW$ (for the quotient quasi-norm on $xV \simeq V/\ker x$) of volume $Q^{\psi(W)-\psi(W\cap \ker x)}$.
On the other hand, the volume of a ball of radius $\eps$ for the restriction of $|\cdot|'$ to $xW\leq E$ is comparable to $\eps^{\phi(xW)}$. Using the triangle inequality {\bf 4.} from the previous subsection, we see that not all $\eps$-balls in $xW$ around the images of  $Q^{\psi(W)}$ integer points can be disjoint if $\eps^{\phi(xW)} Q^{\psi(W)} \gg Q^{\psi(W)-\psi(W\cap \ker x)}$. The proposition follows.
\end{proof}

\begin{remark}Anticipating the next sections, we note here that given any $x \in \Hom(V,E)$, the maps $W \mapsto -\psi(W \cap \ker x)$ and $W \mapsto \phi(xW)$ are both submodular (see $\S \ref{submodstatement}$). This is formal from the submodularity of $-\psi$ and $\phi$. Also they are non-increasing and non-decreasing respectively for set inclusion. So the hypotheses of the submodularity lemma (Lemma \ref{submodular}) are fulfilled.
\end{remark}

\subsection{Pencils}\label{penpen}

Given two non-negative numbers $a,b$, and a subspace $W$ of $V$, we define the \emph{pencil} $\mathcal{P}_{W,a,b}$ associated to our choice of quasi-norms on $V$ and $E$ to be the algebraic subset
$$\mathcal{P}_{W,a,b} =\{x \in \Hom(V,E) \ |\ \psi(\ker x \cap W) \geq a \textnormal{ and }\phi(xW) \leq b\},$$ where $\psi$ (resp. $\phi$) is defined for a subspace $W\leq V$ (resp. $W'\leq E$) by
$$\psi(W)=\sum_{i \in I(W)}  \alpha_i
\quad\mbox{and}\quad \phi(W')=\sum_{i \in J(W')}  \alpha'_i,$$
where as before
\[
I(W)=\{i\in\{1,\ldots,d\} \ |\ \dim(W \cap \langle u_1, \ldots,u_i\rangle) > \dim(W \cap \langle u_1, \ldots,u_{i-1}\rangle)\}
\]
and
\[
J(W')=\{i\in\{1,\ldots,e\} \ |\ \dim(W' \cap \langle u'_{i+1}, \ldots,u'_e\rangle) < \dim(W' \cap \langle u'_i, \ldots,u'_e\rangle)\}.
\]

We easily see that pencils are closed algebraic subvarieties of $\Hom(V,E)$.  Moreover $\{\ker x\ ;\ x \in \mathcal{P}_{W,a,b} , \dim \ker x=k\}$ is a Schubert subvariety of $\Grass_k(V)$ (see $\S \ref{schubert}$).

For an irreducible closed algebraic subset $\cM$ in $\Hom(V,E)$, we define
\begin{equation}\label{tau}
\tau(\mathcal{M})
= \max_{W \leq V} \{ \frac{a}{b} \,;\, \mathcal{M} \subset \mathcal{P}_{W,a,b}\}
\end{equation}
and
\begin{equation}\label{tauq}
\tau_\Q(\mathcal{M})
=\max_{W \leq V,\, \Delta\textnormal{-rational}} \{ \frac{a}{b} \,;\, \mathcal{M} \subset \mathcal{P}_{W,a,b}\}.
\end{equation}

\begin{remark} When all weights $\alpha_i$ and $\alpha_i'$ are equal to $1$, then $\psi$ and $\phi$ are just the dimension functions. Hence in this case our pencil $\mathcal{P}_{W,a,b}$ is just the pencil $\mathcal{P}_{W,r}$ defined in the introduction, with $r$ equal to the integer part of the minimum of $b$ and $\dim W-a$.
\end{remark}

\section{Diophantine approximation and flows on the space of lattices}\label{quasinorm-sec}

In this section we prove the results about diophantine approximation on submanifold of matrices stated in the introduction. We will do so in the general weighted  setting using the quasi-norms defined in the previous section. More precisely we will show the following results (using the notation defined in $(\ref{diophexpdef})$, $(\ref{tau})$ and $(\ref{tauq})$).

\begin{thm}[Exponent for submanifolds of matrices]
\label{uppboundnud}
Let $E,V$ be two real vectors spaces equipped with quasi-norms and $\mathcal{M}=\Phi(U)\subset\Hom(V,E)$ a connected analytic submanifold. Then there is $\widehat{\beta} \in [0,+\infty]$ such that for Lebesgue almost every $x\in U$
$$ \beta(\Phi(x)) = \widehat{\beta}.$$ Moreover $\widehat{\beta}$ depends only on the Zariski closure of $\cM$, and
$$\tau_\Q(\mathcal{M}) \leq \widehat{\beta} \leq \tau(\mathcal{M}).$$
\end{thm}

and

\begin{thm}[Rational manifolds]\label{uppboundnudrat}
In the setting of the previous theorem, assume that the Zariski closure of $\cM$ is defined over $\Q$, then $$\tau_\Q(\mathcal{M})  = \widehat{\beta} = \tau(\mathcal{M}).$$
In particular $ \widehat{\beta}$ is a rational number if the $\alpha_i,\alpha_i'$ are rational.
\end{thm}

The $\Q$-structure on $\Hom(V,E)$ used implicitly in the above theorem is the one induced by the $\Q$-span of the bases of $V$ and $E$ chosen to define the quasi-norms. 

The above statements account for Theorems \ref{rational-exponent-babythm}, \ref{general-inequality} and \ref{weighted} from the introduction. Theorem \ref{pluckerklein} will be proved below as Theorem \ref{inheritance} (together with its consequences \ref{algexp} and \ref{pluckclo}.)

\subsection{The Dani correspondence}\label{subsecdani}

We now recall the connection between diophantine approximation and flows on homogeneous spaces, in particular the so-called quantitative non-divergence estimate, which originates in the work of Margulis \cite{margulis} and Dani \cite{dani1985} and first arose in the groundbreaking work of  Kleinbock and Margulis \cite{kleinbock-margulis} on diophantine approximation on manifolds.

A Dani correspondence is a statement which relates a diophantine exponent to the rate of escape of a certain flow on the space of lattices. In this subsection we present a Dani correspondence for matrices that is valid in the quasi-norm setting. A similar correspondence was worked out by Kleinbock and Margulis already in the matrix context in their work on logarithm laws \cite[Theorem~8.5]{kleinbock-margulis-logarithmlaw}.

We keep the notation of the previous subsection and let $x \in \Hom(V,E)$.
If $I(\ker x)=\{i_1<i_2<\dots<i_n\}$, then
\[ 0<V_{i_1+1}\cap\ker x < V_{i_2+1}\cap\ker x < \dots < V_{i_n+1}\cap\ker x = \ker x\]
is a full flag in $\ker x$.
Similarly, if $J(xV)=\{i_1'<\dots<i_m'\}\subset\{1,\dots,e\}$ is the set of indices $i$ for which $\dim(xV\cap V_{i}')=\dim(V_{i-1}'\cap xV)-1$, then
\[ \ker x < x^{-1}(xV\cap V_{i_m'+1}')<\dots
< x^{-1}(xV\cap V_{i_2'+1}') < x^{-1}(xV\cap V_{i_1'+1}')= V\]
is a full flag from $\ker x$ to $V$.
Concatenating these two flags, we obtain a full flag in $V$:
\[
F_x:\ 0<F_x^{(1)}<\ldots<F_x^{(d)}=V.
\]
We choose an adapted basis $\{e_i^{(x)}\}_i$ of $V$ so that $e_i^{(x)} \in F_x^{(i)} \setminus F_x^{(i-1)}$, and we define a $d$-tuple $(a_1,\ldots,a_d)$ of positive numbers by setting
\[ a_i=
\left\{
\begin{array}{ll}
\alpha_j & \mbox{if}\ F_x^{(i)}= V_{j+1} \cap \ker x\ \mbox{for some}\ j\in I(\ker x)\\
\alpha_j' & \mbox{if}\ F_x^{(i)}=x^{-1}(xV \cap V'_{j+1})\ \mbox{for some}\ j \in J(xV).
\end{array}
\right.\]
Finally given $\beta>0$ we define a one-parameter subgroup $\{\tilde{g}_t^{(x,\beta)}\}_{t\in \R}$ of $\GL(V)$ by
\[
\tilde{g}_t^{(x,\beta)} e_i^{(x)} =
\left\{
\begin{array}{ll}
e^{-a_it} e_i^{(x)} & \hbox{if } i\leq n,\\
e^{\beta a_it} e_i^{(x)} & \hbox{if } i>n,
\end{array}
\right.
\]
where $n=\dim \ker x$. Observe that $a_i \geq a_{i+1}$ for every $i<n$ and $a_i \leq a_{i+1}$ for every $i>n$.

\begin{proposition}[Dani correspondence]
\label{dani}
For $x \in\Hom(V,E)$ the diophantine exponent for quasi-norms defined in $(\ref{diophexpdef})$ is given by
\[ \beta(x)
= \inf\{\beta>0 \ |\ \exists c>0:\ \forall t>0,\ \forall v \in \Delta\setminus\{0\},\ 
\|\tilde{g}_t^{(x,\beta)}v\|\geq c\},\]
where $\|\cdot\|$ is a fixed Euclidean norm on $V$.
\end{proposition}
\begin{proof} We show that given $\beta$ and $x$, $\tilde{g}_t^{(x,\beta)}v$ is uniformly bounded away from zero when $t>0$ and $v \in \Delta\setminus\{0\}$ if and only if $|xv|' \cdot |v|^\beta$ is uniformly bounded away from zero for all non-zero $v \in \Delta$.

Note that $\ker x =\langle e_i^{(x)} , i\leq n\rangle$, and that we may restrict attention to vectors $v$ such that $xv$ is small, for otherwise there is nothing to prove. We may write $v=v'+v''$, where $v'=\sum_{i\leq n} v_i e_i^{(x)}$ and $v''=\sum_{i>n} v_ie_i^{(x)}$ and note, using the triangle inequality 5. of \S\ref{subsecquasi}, that $|v|$ is comparable within multiplicative constants to $|v'|$ since $v''$ can be assumed bounded. So we are left to check that
$$\max\{\max_{i\leq n} |v_i e^{-a_it}| ,  \max_{i> n} |v_i e^{\beta a_it}| \}$$
is bounded away from zero uniformly in $t>0$ and $v \in \Delta \setminus\{ 0\}$ if and only if $|xv|' \cdot |v|^\beta$ is bounded away from zero for $v$ in $\Delta\setminus\{0\}$. This is straightforward since $|xv|'$ is comparable to $\max_{i>n} |v_i|^{1/a_i}$, while $|v|$ is comparable to $|v'|$ and hence to $\max_{i\leq n} |v_i|^{1/a_i}$.
\end{proof}

For the next subsection, it will be convenient to use a slightly different flow $g_t^{(x,\beta)}$ in place of $\tilde{g}_t^{(x,\beta)}$ so as to make the dependence on $x$ explicit while preserving the validity of the Dani correspondence. For this, we fix (independently of $x$) bases $(u_i)$ of $V$ and $(u_i')$ of $E$ that are adapted to the flags $\{V_i\}$ and $\{V_i'\}$, so that $V_i=\langle u_1,\ldots, u_{i-1} \rangle$ and $V'_i=\langle u'_{i+1},\ldots,u'_e\rangle$ and assume $x\in\Hom(V,E)$ is given as a $e\times d$ matrix $(x_{ij})_{i,j}$ in these bases. We denote by $x_i^*$, $1\leq i\leq e$ the rows of $x$, which we view as linear forms on $V$.
The indices in the set $J(xV)=\{i_1'<\dots<i_m'\}$ are obtained by
\[ i_j'=\min\{i\in\{1,\dots,e\} \ |\ \rk(x_1^*,\dots,x_i^*)=j\}.\]
By construction, the family of linear forms $x_{i_m'}^*,\dots,x_{i_1'}^*$ has rank $m=\dim xV$.
We obtain the elements $i_1<\dots<i_n$ of the set $I(\ker x)$ in the following way:
\[ i_j = \min\{ i\in\{1,\dots,d\} \ |\ \rk(u_1^*,\dots,u_i^*,x_{i_m'}^*,\dots,x_{i_1'}^*)=m+j\}.\]
Now, consider the matrix $x'\in\GL_d(\R)$ given in rows by
\[ x' = \begin{pmatrix}
u_{i_1}^*\\
\vdots\\
u_{i_n}^*\\
x_{i_m'}^*\\
\vdots\\
x_{i_1'}^*
\end{pmatrix}.\]
We denote by $a_t^{(\beta)}$ the diagonal matrix $$a_t^{(\beta)}:=\diag(e^{-a_1t},\dots,e^{-a_nt},e^{\beta a_{n+1}t},\dots,e^{\beta a_{n+m}t}),$$ and define $g_t^{(x,\beta)}$ by the matrix
\begin{equation}\label{xxprime}
g_t^{(x,\beta)} := a_t^{(\beta)}x'.
\end{equation}

\begin{remark}
Note that, with this notation, $F_x^{(i)} = x'^{-1}V_{i+1}$ and our former definition $\tilde{g}_t^{(x,\beta)}$ was simply $\tilde{g}_t^{(x,\beta)}=x'^{-1}a_t^{(\beta)}x'$. The Dani correspondence (Proposition~\ref{dani} above) still holds for $\tilde{g}_t^{(x,\beta)}$ in place of $g_t^{(x,\beta)}$, because the shortest vectors of $x'^{-1}a_t^{(\beta)}x'\Delta$ and $a_t^{(\beta)}x'\Delta$ have comparable sizes, up to a positive multiplicative constant depending on $x$ only and not on $t$. 
\end{remark}
With this construction, the map $x\mapsto g_t^{(x,\beta)}$ is a polynomial map on any set where $I(\ker x)$ and $J(xV)$ are constant.
This last condition is equivalent to requiring that $\ker x$ and $xV$ respectively stay in some a fixed Schubert cell of the grassmannians $\Grass(V)$ and $\Grass(E)$, defined with respect to the flags $\{V_i\}$ and $\{V_i'\}$.

\subsection{Quantitative non-divergence}\label{subsecquant}

For the remainder of this subsection, $V$ and $E$ are two filtered real vector spaces endowed with quasi-norms as introduced in \S\ref{dirich}.
In this context, we adapt the strategy introduced by Kleinbock and Margulis \cite{kleinbock-margulis} to study diophantine approximation on manifolds.
The results presented in this subsection are closely related to further work of Kleinbock \cite{kleinbock-anextension,kleinbock-dichotomy}, who proved in particular the existence of an almost sure exponent and the fact that it is the smallest exponent.

Let $\nu$ be a Borel measure on a metric space $X$.
Given an open set $U \subset X$, the measure $\nu$ is \emph{$D$-doubling} (or \emph{$D$-Federer})  on $U$ if every ball $B$ centered at $x \in U \cap \Supp(\nu)$ and contained in $U$  satisfies
\begin{equation}\label{doublingdef}
\nu(\tfrac{1}{3}B) \geq \tfrac{1}{D} \nu(B),
\end{equation}
where $\frac{1}{3}B$ is the ball centered at the center of $B$ whose radius is a third of the radius of $B$.
Given two positive constants $C$ and $\alpha$, we say that a real-valued function $f$ on $X$ is \emph{$(C,\alpha)$-good} on $U$ with respect to the measure $\nu$ if it satisfies, for any ball $B\subset U$ and any $\eps>0$,
\begin{equation}\label{locgooddef}
\nu(\{x\in B \,;\, |f(x)|\leq\eps \}) \leq C \pa{\frac{\eps}{ \|f\|_{\nu,B}}}^{\alpha}\nu(B),
\end{equation}
where $\|f\|_{\nu,B} = \sup_{x\in B\cap\Supp\nu}|f(x)|$. By convention, we also agree that if $f$ is identically zero on the support of $\nu$, then it is $(C,\alpha)$-good with respect to $\nu$ for any values of $C$ and $\alpha$.

The heart of the Kleinbock-Margulis approach to study diophantine approximation on manifolds is the following key result, which can be seen as a Remez-type inequality (see \cite{brudnyi-ganzburg} or \cite[Thm 2.7]{abrs}) for functions taking values in the space of lattices. It originates in the work of Margulis \cite{margulis1971} on non-divergence of unipotent flows, which was later greatly generalized in work of Dani \cite{dani1985} and Kleinbock-Margulis \cite{kleinbock-margulis}. The following version is borrowed from \cite{kleinbock-anextension} (see also \cite{kleinbock-pisa}).

\begin{theorem}\label{klw}
\cite[Theorem~2.2]{kleinbock-anextension}
Let $X$ be a Besicovitch metric space. Let $U \subset X$ be an open subset and $\nu$ a Radon measure on $U$ which is $D$-doubling on $U$ for some $D>0$. Let $C,\alpha>0$ be positive constants and $h:U \to \rm \GL_d(\R)$ be a continuous map. Let $\rho \in (0,1]$ and let $B=B(z,r)$, $z \in \Supp(\nu)$, be an open ball such that $B(z,3^dr)$  is contained in $U$. Assume that for each $v_1,\ldots, v_k$ in $\Z^d\setminus\{0\}$, $1\leq k\leq d$,

\begin{itemize}
\item the function $x \mapsto \|h(x)\mathbf{w}\|$ is $(C,\alpha)$-good for $\nu$ on $B(z,3^dr)$, and
\item if $\mathbf{w}\neq 0$, then $\sup_{y \in B \cap \Supp(\nu)} \|h(y) \mathbf{w}\| > \rho^k$,
\end{itemize}
where $\bw=v_1\wedge v_2\wedge\dots\wedge v_k$.\\
Then  for every $\eps \in (0,\rho]$, we have that
\[ \nu(\{x \in B ; h(x)\Z^d \in \Omega_\eps \}) \leq C'\big(\tfrac{\eps}{\rho}\big)^\alpha \nu(B), \]
where $C'>0$ depends only on $D,C,\alpha$ and the Besicovitch constant of $X$.
\end{theorem}

In the above theorem, the letter $\Omega_\eps$ denotes the part of the space of lattices $\rm \GL_d(\R)/\GL_d(\Z)$ made of lattices in $\R^d$ admitting a non-zero vector of length at most $\eps$. When one restricts attention to unimodular lattices, the set $\Omega_\eps$ is thought of as \enquote{the cusp} of $\Omega$, because as $\eps \to 0$ these sets form a nested sequence of neighborhoods of infinity. This fact is known as Mahler's criterion \cite[Cor. 10.9]{raghunathan}.

In fact the theorem is stated only for $\SL_d(\R)$-valued maps in \cite{kleinbock-anextension}, but it is valid as well -- with the same proof -- for $\GL_d(\R)$-valued maps. In older versions of this non-divergence result (such as in \cite{kleinbock-lindenstrauss-weiss}) the second assumption involved a lower bound of the form $\rho$. Kleinbock's observation \cite{kleinbock-anextension} that one can relax this assumption to a lower bound $\rho^k$ is crucial when one wants to study diophantine exponents of measures that might not be extremal; it will be essential in the proofs of our formulas for the exponent.

Recall also that a metric space $X$ is called \emph{Besicovitch} if there exists an integer $C$ such that we have the following property: suppose $A\subset X$ and for each $a\in A$ we are given a non-empty open ball $B_a$ centered at $a$; then there exists a countable subset $A'\subset A$ such that $A\subset\bigcup_{a\in A'}B_a$ and any intersection of $C$ distinct balls $B_a$, $a\in A'$, is empty.
In the definition below, the flow $g_t^{(y,\beta)}$ on $\GL(V)$ is the one constructed at the end of \S\ref{subsecdani}, and we denote by $\cW_{\Delta}^k$ the set of non-zero pure integral $k$-vectors, i.e. those $\bw\in\wedge^kV \setminus\{0\}$ that can be written $\bw=v_1\wedge\dots\wedge v_k$ where each $v_i$ is an element of $\Delta$.

\begin{definition}\label{locally-good-def}
Let $\nu$ be a Radon measure on $\Hom(V,E)$ and $x\in\Supp(\nu)$.
We will say that the measure $\nu$ is \emph{locally good at $x$} if there exists a neighborhood $U_x$ of $x$ and positive constants $C$, $D$ and $\alpha$ such that
\begin{itemize}
\item there exists $n \leq d$ and a Schubert cell  $e(\sigma) \subset\Grass_n(V)$ such that for all $y$ in $U_x\cap\Supp\nu$, $\ker y$ belongs to $e(\sigma)$ (see $\S \ref{schubert}$).
\item there exists $m \leq e$ and a Schubert cell $e(\sigma')\subset\Grass_m(E)$ such that for all $y$ in $U_x\cap\Supp\nu$, $yV$ belongs to $e(\sigma')$  (see $\S \ref{schubert}$).
\item $\nu$ is $D$-doubling on $U_x$
\item for all $t,\beta>0$, all $k\geq0$ and all $\bw$ in $\cW_\Delta^k$, the map $y \mapsto \|g_t^{(y,\beta)}\bw\|$ is $(C,\alpha)$-good on $U_x$ with respect to $\nu$.
\end{itemize}
\end{definition}

\begin{remark}
The first two conditions ensure that the map $y\mapsto g_t^{(y,\beta)}$ is continuous on a neighborhood of $x$ in $\Supp\nu$.
\end{remark}

\begin{remark}
The set of points where $\nu$ is locally good is an open subset of $\Supp(\nu)$.
\end{remark}

With the quantitative non-divergence result (Theorem~\ref{klw}) we can derive the following statement.
The argument in the proof is taken from Kleinbock \cite{kleinbock-dichotomy}.

\begin{theorem}[Existence of a local exponent]\label{localexponent}
If a Radon measure $\nu$ on $\Hom(V,E)$ is locally good at $x$, then there is a neighborhood $B_x$ of $x$ such that for $\nu$-almost every point $y \in B_x$,
\begin{equation}\label{locexp}
\beta(y) = \inf_{z\in B_x} \beta(z). 
\end{equation}
If $B_x$ is a neighborhood such that (\ref{locexp}) holds, we define the \emph{local diophantine exponent} of $\nu$ at $x$ by
\[ \widehat{\beta}_\nu(x) = \inf_{z\in B_x}\beta(z).\]
\end{theorem}
\begin{proof}
If $B_x$ is any neighborhood of $x$, it is trivial that for all $y$ in $B_x$, $\beta(y)\geq\inf_{z\in B_x}\beta(z)$.
Conversely, assume $\nu$ is locally good at $x$, and let $C, D, \alpha>0$ and $B_x=B(x,r)$ be a ball around $x$ such that the conditions of Definition \ref{locally-good-def} hold on $U_x:=B(x,3^dr)$. We claim that for almost every $y$ in $B_x$, $\beta(y)\leq\inf_{z\in B_x}\beta(z)$.
To see this, fix $z\in B_x$ and $\beta>\beta(z)$, so that, by Proposition~\ref{dani}, there exists $c>0$ such that
\[ \forall t>0,\ \forall k,\ \forall\bw\in\cW_\Delta^k,\quad
\|g_t^{(z,\beta)}\bw\| \geq c.\]
Of course, this implies that
\[ \forall t>0,\ \forall k,\ \forall\bw\in\cW_\Delta^k,\quad
\sup_{y\in B_x} \|g_t^{(y,\beta)}\bw\| \geq c.\]
For $\eta>0$, apply Theorem~\ref{klw} with $U=U_x$, $B=B_x$, $\rho=c$, $h:y\mapsto g_t^{(y,\beta)}$, and $\eps=e^{-\eta t}$.
For any $t>0$ such that $e^{-\eta t}<c$, we find that
\[ \nu(\{y\in B_x\ |\ g_t^{(y,\beta)}\Delta\in\Omega_{e^{-\eta t}}\})
\leq C' \big(\frac{e^{-\eta t}}{c}\big)^\alpha\nu(B_x).\]
Therefore,
\[ \sum_{t\in\N} \nu(\{y\in B_x\ |\ g_t^{(y,\beta)}\Delta\in\Omega_{e^{-\eta t}}\}) < \infty\]
and, by the Borel-Cantelli lemma and Proposition~\ref{dani} again,
for almost every $y$ in $B_x$, $\beta(y)\leq\beta+\eta$.
Letting $\eta\to 0$, $\beta\to\beta(z)$, and taking the infimum over $z$ in $B_x$, we get, for almost every $y$ in $B_x$,
\[ \beta(y) \leq \inf_{z\in B_x} \beta(z).\]
\end{proof}

Given a non-negative parameter $\beta$, we say that the measure $\nu$ satisfies the condition (\ref{cbeta}) at $x$ if the following holds:
\begin{equation}\label{cbeta}
\tag{$\mathcal{C}_\beta$}
\begin{array}{c}
\forall r>0,\
\exists c>0:\
\forall k\in\{1,\dots,\dim V\},\ \forall t>0,\ \forall \bw\in\mathcal{W}_{\Delta}^k,\\
\sup_{y \in B(x,r)\cap\Supp\nu} \| g_t^{(y,\beta)}\bw\| \geq c.
\end{array}
\end{equation}
It follows from Theorem~\ref{localexponent} and Proposition~\ref{dani} that if $\nu$ is locally good at $x$, then
\begin{equation}\label{betainf}
\widehat{\beta}_\nu(x) = \inf\{\beta>0 \,|\, \nu\ \mbox{satisfies}\ (\mathcal{C}_\beta)\ \mbox{at}\ x\}.
\end{equation}
We now use this observation to show that the local exponent $\widehat{\beta}_\nu(x)$ of a locally good measure at $x$ is determined by the local Zariski closure at $x$ of the support of $\nu$.
More precisely, fix two bases for $V$ and $E$, and given $x$ in $\Hom(V,E)$, consider $\theta(x)\in\R^N$ the $N$-tuple of all minors of $x$, where $N=\sum_{k=1}^e{d\choose k}{e\choose k}={e+d\choose d} -1$.
For a subset $S$ in $\Hom(V,E)$, we define \emph{the Pl\"ucker closure} of $S$ by 
\[ \cH(S) = \{ x\in\Hom(V,E) \ |\ \theta(x)\in \Span(\theta(y)\,;\,y\in S)\},\]
and finally, if $\nu$ is a locally good measure at $x$ in $\Hom(V,E)$, we let
\begin{equation}\label{locpluck} \cH_\nu(x) = \bigcap_{r>0} \cH(B(x,r)\cap\Supp\nu)\end{equation}
be the \emph{local Pl\"ucker closure at $x$.}

\begin{theorem}[Inheritance]
\label{inheritance}
Let $\nu$, $\nu'$ be Radon measures on $\Hom(V,E)$, and assume that $\nu$ is locally good at $x$ and $\nu'$ locally good at $x'$.
If $\cH_\nu(x)\subset\cH_{\nu'}(x')$, then $\widehat{\beta}_\nu(x)\geq\widehat{\beta}_{\nu'}(x')$.
\end{theorem}
\begin{proof}
Let $S$ be a compact subset of $\Hom(V,E)$, and assume that for all $x$ in $S$, the subspaces $\ker x$ and $xV$ belong to some fixed Schubert cells in $\Grass(V)$ and $\Grass(E)$.
We will say that $S$ satisfies (\ref{cbeta}) if 
\begin{equation}\label{cbetas}
\tag{$\mathcal{C}_\beta(S)$}
\begin{array}{c}
\exists c>0:\
\forall k\in\{1,\dots,\dim V\},\ \forall t>0,\ \forall \bw\in\mathcal{W}_{\Delta}^k,\\
\sup_{y \in S} \| g_t^{(y,\beta)}\bw\| \geq c.
\end{array}
\end{equation}
The expression $g_t^{(y,\beta)}\bw$ is a linear function of the minors of $g_t^{(y,\beta)}$.
By (\ref{xxprime}), it is also a linear function of the minors of $y$, and therefore, for some constant $C>0$ depending only on $S$,
\[ \sup_{y \in S} \| g_t^{(y,\beta)}\bw\|
\leq \sup_{y \in \cH(S)\cap B} \| g_t^{(y,\beta)}\bw\|
\leq C\cdot\sup_{y \in S} \| g_t^{(y,\beta)}\bw\|,\] where $B$ is a fixed compact neighborhood of $0$ in $\Hom(V,E)$ containing $S$ (note that both expressions on the right-hand side are norms on the finite-dimensional space of linear maps between $\theta(\mathcal{H}(S))$ and $\wedge^k V$).
It follows that the condition ($\mathcal{C}_\beta(S)$) depends only on $\cH(S)$.
To conclude the proof of the theorem, it suffices to use equality (\ref{betainf}), and to note that the measure $\nu$ satisfies (\ref{cbeta}) at $x$ if and only if ($\mathcal{C}_\beta(B(x,r)\cap\Supp\nu)$) holds for every $r>0$.
\end{proof}

\begin{remark} Note that if $S$ is a subset of $\Hom(V,E)$, then $\mathcal{H}(S)$ is a Zariski closed subset containing $S$, so in particular it contains the Zariski closure $Z(S)$ of $S$,  i.e. the intersection of all closed real algebraic subsets of $\Hom(V,E)$ containing $S$. Since by definition   $\mathcal{H}(S)=\mathcal{H}(\mathcal{H}(S))$ we also get $\mathcal{H}(S)=\mathcal{H}(Z(S))$.
\end{remark}

\begin{remark}[local Zariski closure]
The local Zariski closure of $\nu$ at $x$ is defined to be the intersection of the Zariski closures of $B(x,r) \cap \Supp(\nu)$ in $\Hom(V,E)$ for all $r>0$. By noetherianity the local Zariski closure (resp. the local Pl\"ucker closure) coincides with the Zariski closure (resp. the Pl\"ucker closure) of $B(x,r) \cap \Supp(\nu)$ whenever $r$ is sufficiently small. Observe in particular that $\cH_\nu(x)$ depends only on the local Zariski closure at $x$.
\end{remark}

Hence we obtain the following corollary:

\begin{corollary}\label{localZar} Let $\nu$, $\nu'$ be Radon measures on $\Hom(V,E)$, and assume that $\nu$ is locally good at $x$ and $\nu'$ locally good at $x'$. If the local Zariski closures coincide, then $\widehat{\beta}_\nu(x)=\widehat{\beta}_{\nu'}(x')$.
\end{corollary}

As a corollary of this inheritance theorem, we can define the diophantine exponent of an algebraic subset of $\Hom(V,E)$.
Here, and throughout the paper, when speaking of Zariski closure, algebraic subsets and algebraic varieties, we will always consider these notions in real algebraic geometry and we refer the reader to the textbook \cite{bochnak-coste-roy} for definitions and basic properties.
By Lebesgue measure on an algebraic set $\cM$ we mean the top-dimensional Hausdorff measure on the subset $\cM$ of $\Hom(V,E)$.

\begin{corollary}[Exponent of an algebraic subset]
\label{algexp}
Let $\cM$ be an irreducible algebraic subset in $\Hom(V,E)$.
There exists $\widehat{\beta}(\cM)$ such that, for almost every $x$ in $\cM$ with respect to the Lebesgue measure,
\[ \beta(x) = \widehat{\beta}(\cM).\]
\end{corollary}
\begin{proof}
Let $x$ be a non-singular point on $\cM$ such that $\ker x$ and $xV$ lie in the interior of the smallest Schubert subvariety (of $\Grass(V) $ and $\Grass(E)$ respectively) containing $\{\ker y\,;\, y\in\cM\}$ and $\{yV\,;\, y\in\cM\}$. Since $\cM$ is irreducible, by \cite[\S3.2]{bochnak-coste-roy}, this makes a set of full measure in $\cM$.
Take an analytic parametrization $\Phi:U\to\Hom(V,E)$ of $\cM$ in a neighborhood of $x$, and let $\nu$ be the pushforward under $\Phi$ of the Lebesgue measure on $U$.
Note that $\nu$ is equivalent to the Lebesgue measure on $\cM$ in a neighborhood of $x$.
Restricting $U$ if necessary, we see using \cite[Proposition~2.1]{kleinbock-dichotomy} that there exist constants $(C,\alpha)$ and a neighborhood $U_x$ of $x$ such that all functions $y\mapsto \|g_t^{(y,\beta)}\bw\|$ are $(C,\alpha)$-good on $U_x$ with respect to $\nu$.
So $\nu$ is locally good at $x$, and by Theorem~\ref{localexponent}, for $\nu$-almost all $y$ in a neighborhood of $x$, $\beta(y)=\widehat{\beta}_\nu(x)$. At a non-singular $x$ as above the local Zariski closure is $\cM$ itself (\cite[Proposition~3.3.14]{bochnak-coste-roy}). So by Corollary~\ref{localZar}, $\widehat{\beta}_\nu(x)$ is independent of the choice of $x$. This proves the corollary.
\end{proof}

\begin{remark}\label{pluckclo} We can also conclude that $\widehat{\beta}(\cM) = \widehat{\beta}(\mathcal{H(\cM)})$. Indeed the local Pl\"ucker closure at a random point of $\cM$ is $\mathcal{H}(\cM)$, and so is the local Pl\"ucker closure at a random point of $\mathcal{H}(\cM)$. We actually believe that more is true, namely that  $\widehat{\beta}(\cM) = \widehat{\beta}(\mathcal{S(\cM)})$, where $\mathcal{S(\cM)}$ is the \emph{Schubert closure} of $\cM$, namely the intersection of all pencils containing $\cM$. 
\end{remark}

The pushforward of the Lebesgue measure under an analytic map $U \to\Hom(V,E)$, where $U$ is a open domain in $\R^N$ is locally good at every point, so we have the following important corollary to Theorem~\ref{localexponent}.
This result is due to Kleinbock \cite{kleinbock-dichotomy} in the case where $|\cdot|$ and $|\cdot|'$ are norms, and the proof is essentially the same, once the correspondence of \S\ref{subsecdani} has been established.

\begin{corollary}[Exponent of an analytic manifold]
\label{analytic-existence}
Let $\Phi:U\to\Hom(V,E)$ be an analytic map on a connected open set $U\subset\R^N$.
There exists $\widehat{\beta}\in[0,\infty]$ such that, for almost every $x$ in $U$,
\[ \beta(\Phi(x))=\widehat{\beta}.\] Moreover $\widehat{\beta}=\widehat{\beta}(\mathcal{M})$, where $\mathcal{M}$ is the Zariski closure of $\{\Phi(u) \,;\, u\in U\}$.
\end{corollary}
\begin{proof}
Since $\Phi$ is analytic, $\cM$ is irreducible, and there exists a subset of full measure $U'\subset U$ such that, for all $x$ in $U'$, the local Zariski closure at $x$ (i.e. the intersection of the Zariski closures of all neighborhoods of $x$) is equal to $\cM$.
For such an $x$ the pushforward under $\Phi$ of the Lebesgue measure on a neighborhood of $x$ is locally good at $x$ (\cite[Proposition~2.1]{kleinbock-dichotomy}) so that $\beta(\Phi(y))=\widehat{\beta}(\cM)$ for almost every $y$ in a neighborhood of $x$.
\end{proof}

\subsection{Pencils and extremality}
Given two non-negative numbers $a,b$ and a subspace $W$ of $V$, recall that we have defined the \emph{pencil} $\mathcal{P}_{W,a,b}$ associated to our choice of quasi-norms on $V$ and $E$ to be the set
$$\mathcal{P}_{W,a,b}=\{x \in \Hom(V,E) \ |\ \psi(\ker x \cap W) \geq a \textnormal{ and }\phi(xW) \leq b\}.$$
And the numbers $\tau(\cM)$ and $\tau_\Q(\cM)$ have been defined in $(\ref{tau})$ and $(\ref{tauq})$.

\begin{theorem}\label{uppboundnu}Let $E,V$ be two real vector spaces equipped with quasi-norms and $\mathcal{M}\leq\Hom(V,E)$ an irreducible algebraic subset. Then
$$\tau_\Q(\mathcal{M}) \leq \widehat{\beta}(\mathcal{M}) \leq \tau(\mathcal{M}).$$
\end{theorem}
\begin{proof}
The left-hand side follows from Dirichlet's principle proved in Proposition \ref{dirichlet}.
Let us verify the right-hand side.
Let $\beta>\tau(\mathcal{M})$.
By Theorem~\ref{localexponent}, it suffices to show that $\beta>\widehat{\beta}_\nu(x)$ for every non-singular point $x$ of $\mathcal{M}$ which lies in the interior of the smallest Schubert subvariety containing $\cM$. Equivalently, from $(\ref{betainf})$, it is enough to check that $(\mathcal{C}_\beta)$ holds at $x$. Note that we may replace $g_t^{(x,\beta)}$ by $\tilde{g}_t^{(x,\beta)}$ in the definition of $(\mathcal{C}_\beta)$, because $x'$ remains bounded in a neighborhood of $x$ in $\cM$.
Clearly then a sufficient condition for $(\mathcal{C}_\beta)$ to hold is that $\bw \mapsto \sup_{y \in B(x,r)\cap \mathcal{M}}\|\pi_0^{(y,\beta)}(\bw)\|$ does not vanish for any $r>0$, when $\bw$ ranges among pure $k$-vectors of norm $1$, and $\pi_0^{(y,\beta)}$ is the projection onto the sum of the eigenspaces of $\tilde{g}_t^{(y,\beta)}$ with non-negative eigenvalue with kernel the sum of the other eigenspaces.

Now, if $W$ is the subspace of $V$ associated to the $k$-vector $\bw$, then the largest eigenvalue occurring in the decomposition of $\bw$ into eigenvectors of $\tilde{g}_t^{(y,\beta)}$ is $\beta\phi(yW)-\psi(\ker x\cap W)$.
Therefore, if $(\mathcal{C}_\beta)$ fails to hold, there exist $r>0$ and a subspace $W\leq V$ such that $B(x,r) \cap \mathcal{M}$ is entirely contained in the set
$$ \{y \in \Hom(V,E) \ |\  \beta\phi(yW) - \psi(\ker y \cap W) <0 \},$$
which is
the (finite) union of all pencils $\mathcal{P}_{W,a,b}$ such that $\beta b - a <0$. Recall that neighborhoods of non-singular points are Zariski-dense in $\cM$ \cite[Proposition~3.3.14]{bochnak-coste-roy}. From the irreducibility of $\mathcal{M}$ it follows that $\mathcal{M}$ is entirely contained in a single pencil $\mathcal{P}_{W,a,b}$ for a pair $a,b$ with $\beta b - a <0$, which is contrary to our assumption that $\beta>\tau(\nu)$.
\end{proof}

\begin{proof}[Proof of Theorem \ref{uppboundnud}] Combine Corollary \ref{analytic-existence} with Theorem \ref{uppboundnu}. 
\end{proof}

\section{The submodularity lemma}\label{section:submodular}

In view of Theorem~\ref{uppboundnu}, the following question is natural: Under what condition on the irreducible algebraic subset $\cM\subset\Hom(V,E)$ is the maximum defining $\tau(\cM)$ in $(\ref{tau})$ attained on a $\Delta$-rational subspace $W$?
We will show here that a sufficient condition is that $\cM$ be defined over the rationals. This is the content of Theorem~\ref{rathm}.
The method will also show that if $\cM$ is invariant under some group $G$ of linear automorphisms, then $\tau(\cM)$ is attained on a $G$-invariant subspace $W$. This observation will be essential for the application to nilpotent Lie groups developed in Section~\ref{section:critical}.

\subsection{Statement and proof of the submodularity lemma}\label{submodstatement}

Let $V$ be a $d$-dimensional vector space (over some field). Let $\phi$ and $\psi$ be two real-valued functions on the Grassmannian of $V$, with the following properties:

\begin{enumerate}
\item $\phi \geq 0$, $\phi(0)=0$.
\item $\phi$ and $\psi$ are non-decreasing for set inclusion.
\item $\phi$ and $-\psi$ are \emph{submodular}, i.e. for any two vector subspaces $U$ and $W$ we have
\[ \phi(U+ W)+\phi(U\cap W) \leq \phi(U)+\phi(W), \]
\[ \psi(U+ W)+\psi(U\cap W) \geq \psi(U)+\psi(W). \]
\end{enumerate}

We let
\[ q(x) = \frac{\psi(x)}{\phi(x)} \]
with the convention that $q(x)=0$ if $\psi(x)=0$ and $q(x)=\sign(\psi(x)) \infty$ if $\phi(x)=0$ but $\psi(x) \neq 0$. We are interested in the supremum $S$ of $q$ on the entire Grassmannian. Note that $S \in \R \cup \{\pm \infty\}$.

\begin{lemma}[Submodularity lemma]\label{presub}  The supremum
$$S = \sup_{x \in \Grass(V)} q(x)$$
is attained, and there is a unique subspace $x_0$ of maximal dimension such that
$$S=q(x_0).$$
\end{lemma}

\begin{proof} We perform a preliminary reduction: we are going to reduce to the case where $\phi(x)>0$ when $x \neq 0$ and $S>q(0)$. For this we make the following initial observation. Let $W$ be the sum of all subspaces $x$ such that $\phi(x)=0$. Then  $\phi(W)=0$. Indeed, using that $\phi$ is non-decreasing and submodular,  if $\phi(x)=\phi(y)=0$, then $0 \leq \phi(x \cap y) \leq \phi(x) = 0$, so that $\phi(x \cap y)=0$, while $0 \leq \phi(x+y) \leq \phi(x)+\phi(y) = 0$, so $\phi(x+y)=0$.

It follows that for every $x \in \Grass(V)$, $\phi(x+W)=\phi(x)$. Indeed $\phi(x) \leq \phi(x+W) \leq \phi(x)+\phi(W)=\phi(x)$. Also $\psi(x+W) \geq \psi(x)$. So $q(x+W) \geq q(x)$. In particular $S =\sup_{x \supset W} q(x)$. We may thus restrict $\phi$ and $\psi$ to those subspaces of $V$ that contain $W$. This gives rise to new functions $\phi'$ and $\psi'$ on the quotient space $V/W$, which clearly satisfy the same properties.

So we have reduced the proof to the case where $W=0$, i.e. $\phi(x)>0$ unless $x=0$. If $\psi(0) > 0$, then $q(0)=+\infty$, while $q(x)<+\infty$ if $x \neq 0$, so the conclusion of the lemma holds with $x_0=0$. Therefore we may assume that $\psi(0)\leq 0$. This implies that $q(0)<S$, provided $\psi$ is not identically zero.  Indeed, if $\psi(0)<0$, then $q(0)=-\infty$, while $q(x) > -\infty$ if $x \neq 0$, so $S>q(0)$. While if $\psi(0)=0$, then $q(0)=0$, while $S>0$ provided $\psi$ is not identically zero. If $\psi$ is identically zero, the conclusion of the lemma holds with $x_0=V$.

So we have proved our initial claim and we thus assume without loss of generality that $\psi$ is not identically zero, that $q(0)<S$, and that $\phi(x)>0$ if $x \neq 0$. This enables us to assert that $S=S_1$, where we have denoted, for  $k=1,\ldots,d$, $$S_k =\sup\{q(x) ; \dim x \geq k\}.$$

Let $k_0$ be the maximal $k\geq 1$ such that $S_k=S$.
If $k_0=d$, then $S=q(V)$ and the conclusion of the lemma holds with $x_0=V$.
If not we have $S_{k_0+1} < S$.

Pick $\eps>0$ such that $S-S_{k_0+1} >2\eps$, and pick $x_0 \in \Grass(V)$ such that $\dim(x_0) \geq k_0$ and $q(x_0) > S - \eps$.
Note that $\dim(x_0) = k_0$, for otherwise $\dim(x_0)\geq k_0+1$ and thus $q(x_0) \leq S_{k_0+1} < S-2\eps < S-\eps$, a contradiction to our choice of $x_0$.

Now let $y_0$ be another choice of subspace such that $\dim(y_0) \geq k_0$ and $q(y_0) > S - \eps$.
For the same reason $\dim(y_0)=k_0=\dim(x_0)$.
But
\begin{align*}
\psi(x_0 + y_0) & \geq \psi(x_0) + \psi(y_0) - \psi(x_0 \cap y_0)\\
& \geq (S-\eps)(\phi(x_0) + \phi(y_0)) - S\phi(x_0 \cap y_0)\\
& \geq S(\phi(x_0) + \phi(y_0)- \phi(x_0 \cap y_0)) -\eps(\phi(x_0) + \phi(y_0)) \\
& \geq S\phi(x_0 +y_0) -\eps(\phi(x_0) + \phi(y_0)) \\
& \geq S\phi(x_0 +y_0) -2\eps(\phi(x_0+y_0)) = (S-2\eps) \phi(x_0+y_0)
\end{align*}
where we have used submodularity of $-\psi$ in the first line, positivity of $\phi$ in the second line, submodularity of $\phi$ in the fourth line and monotonicity of $\phi$ in the last line. Hence:
\[
q(x_0+y_0) \geq S-2\eps > S_{k_0+1}.
\]
Therefore $\dim(x_0+y_0) \leq k_0$. This means that $x_0=y_0$. Hence we have proved that $q(x)>S-\eps$ and $\dim(x)\geq k_0$ implies that $x$ is unique. In particular $S=q(x_0)$, and the lemma follows.
\end{proof}

\begin{corollary}\label{submodular}
Let $G$ be a group acting on the Grassmannian, and assume that the action preserves dimension.
If $\phi$ and $\psi$ are invariant under $G$, then the supremum $S$ is attained on a $G$-invariant subspace.
\end{corollary}
\begin{proof}
Indeed $x_0$ and $gx_0$ will have the same dimension and will both achieve the supremum $S$, so by uniqueness $x_0=gx_0$, for every $g \in G$.
\end{proof}

\begin{remark} The proof works verbatim more generally for functions defined on a graded lattice of finite length, in place of $\Grass(V)$, i.e. a partially ordered set with a smallest element (0) and a largest element (V) such that every pair of elements admits unique lower and upper bounds, and which is equipped with an integer-valued rank function $r(x)$, which is $\geq 0$ and takes only finitely many values and is such that if $x < y$ and there is no $z$ with $x < z < y$, then $r(y)=r(x)+1$.
\end{remark}

\subsection{Applications of the submodularity lemma}

We go back to the setting of Section~4.
Thus, $V$ and $E$ are finite-dimensional real vector spaces, endowed with quasi-norms $|\cdot|$ and $|\cdot|'$ with associated flags $\{V_i\}$ and $\{V_i'\}$.
The lattice $\Delta \leq V$ induces a $\Q$-structure on $V$. Recall that by a $\Q$-structure we mean a $\Q$-vector subspace which generates the ambient  space over $\R$ and whose dimension over $\Q$ is the dimension over $\R$ of the ambient space.
Let us assume that the flag $\{V_j\}$ is made of $\Delta$-rational subspaces,
and endow $E$ with a $\Q$-structure for which each $V'_j$ is rational.
This endows $\Hom(V,E)$ with a natural $\Q$-structure.

\begin{theorem}[Exponent of a rational algebraic subset]
\label{rathm}
Assume that the irreducible algebraic subset $\mathcal{M}\leq \Hom(V,E)$ is defined over $\Q$. Then
\[ \widehat{\beta}(\mathcal{M})=\tau(\mathcal{M})=\tau_\Q(\mathcal{M})=\max_{W\subset V,\,\Delta\textnormal{-rational}}\{\frac{a}{b}\,;\, \cM\subset\cP_{W,a,b}\}.\]
\end{theorem}

\begin{proof}
From Theorem~\ref{uppboundnu}, it is enough to show that $\tau(\cM)=\tau_\Q(\cM)$.
By definition
\begin{equation}\label{tausub}
\tau(\cM) = \max_{W \leq V} \frac{\psi_\mathcal{M}(W)}{\phi_\mathcal{M}(W)}
\quad\mbox{and}\quad
\tau_\Q(\cM) = \max_{W \leq V, W \Delta\textnormal{-rational}} \frac{\psi_\mathcal{M}(W)}{\phi_\mathcal{M}(W)}.
\end{equation}
where
\begin{equation}\label{psiandphi}
\psi_\mathcal{M}(W) = \min_{x \in \mathcal{M}} \psi(W \cap \ker x)
\quad\mbox{and}\quad
\phi_\mathcal{M}(W) =\max_{x \in \mathcal{M}} \phi(xW).
\end{equation}
These functions are defined a priori on the real grassmannian $\Grass(V)$, but they also make sense on the grassmannian of the complexified space $V^{\C}$, where $x$ is now viewed as a linear map from $V^{\C}$ to $E^{\C}$, and the flags $\{V_j\}$ and $\{V_i'\}$ are also complexified.
The Galois group $\Gal(\C|\Q)$ acts on $\Grass(V^{\C})$ and preserves $\psi_\mathcal{M}$ and $\phi_\mathcal{M}$, because $\mathcal{M}$ is defined over $\Q$ as well as the flags $\{V_j\}$ and $\{V_i'\}$.
The functions $\psi_\mathcal{M}$ and $\phi_\mathcal{M}$ are both non-decreasing and non-negative functions.
Moreover $-\psi_\mathcal{M}$ and $\phi_\mathcal{M}$ are submodular.
This follows from the submodularity of $W \mapsto -\psi(W \cap \ker x)$ and $W \mapsto \phi(xW)$ for a given $x$ proved in Subsection \ref{subsecquasi}, from the irreducibility of $\mathcal{M}$ and from the fact that for each $W$ there is a Zariski-open set of $x$ in $\mathcal{M}$ such that $\psi_\mathcal{M}(W)=\psi(W \cap \ker x)$ (resp. $\phi_\mathcal{M}(W)=\phi(xW)$).
We may thus apply the submodularity lemma (Lemma \ref{submodular}) and conclude that the maximum in the definition of $\tau(\mathcal{M})$ is attained for a rational $W$.
\end{proof}

This completes the proof of Theorem \ref{uppboundnudrat}. A different application of the submodularity lemma will be given in Section~\ref{section:critical}, in the proof of Theorem~\ref{rationality-stability} from the introduction.

For now, we just give another simple example where the submodularity lemma applies, and allows to compute the diophantine exponent of an algebraic subset of matrices.
For the example below, and until the end of this section, we shall only consider diophantine approximation for genuine norms on $V$ and $E$.

\begin{example}[The Veronese curve in algebras]\label{veronese}
Consider the Veronese curve $\cV$ in $\R^n$ given by the parametrization $x\mapsto (x,\dots,x^n)$.
The Mahler conjecture proved by Sprind\u{z}uk says that the curve $\cV$ is extremal, or in other words, that for almost every $x\in\R$, for all $\eps>0$, the inequality
\begin{equation}\label{vero}
|v_0 + v_1 x + \dots + v_n x^n| \geq \|v\|^{-(\beta+\eps)}
\end{equation}
has only finitely many integer solutions $v\in\Z^{n+1}$, where here $\beta=n$.

We can consider the analogous problem, where $\R$ is replaced by an arbitrary finite-dimensional $\R$-algebra $E$ with unit (e.g. $E=M_m(\R)$). Given $x \in E$, we may consider integer linear combinations of $1,x,\dots,x^n$. We may then ask for  the minimal $\beta>0$ such that for almost every $x\in E$ and for every $\eps>0$, the inequality $(\ref{vero})$ (with a norm in place of the absolute value on the left-hand side) has at most finitely many integer solutions $v\in\Z^{n+1}$.

To fit this problem into our setting, we let $V=\R_n[X]$ the space of polynomials of degree at most $n$ and $\Delta$ the lattice of polynomials with integer coefficients.
Consider the submanifold $\cM$ of $\Hom(V,E)$ that consists of the evaluations maps $P \mapsto P(x)$ for $x\in E$.
It is defined over $\Q$, so, by Theorem \ref{rathm}, its exponent, which is exactly the $\beta$ we are looking for, is equal to
$$\tau(\cM) = \max \set{ \frac{\psi(W)}{\phi(W)} ; \, W\leq V }, $$
where
\[ \phi(W)=\max_{x\in E} \dim W(x)
\quad\mbox{and}\quad
\psi(W)=\min_{x\in E} \dim W-\dim W(x).\]
Here $W(x)\leq E$ denotes the image of $W$ under the evaluation map $P \mapsto P(x)$.
Now, let $G$ be the group of affine transformations of the real line and let it act on $V$ by substitution of the variable.
The maps $\phi$ and $-\psi$ are submodular and $G$-invariant, so by Lemma \ref{submodular}, $\tau$ achieves its value on one of the $G$-invariant subspaces, which are exactly the subspaces $V_i\leq V$, $i=0,\dots,n$ of polynomials of degree at most $i$.
Let $m$ be the maximal dimension of one-generator subalgebras $\R[x] \leq E$, for $x \in E$.
Evaluating $V_i$ at an $x$ with minimal polynomial of degree $m$, we see that, for $i<m$,  $\phi(V_i)=i+1$ and $\psi(V_i)=0$.
On the other hand, we always have $\phi(W)\leq m$ and $\psi(W)\geq\dim W-m$, so we find
\[ \frac{\psi(V_i)}{\phi(V_i)} =
\left\{ \begin{array}{ll}
0 & \mbox{if}\ i< m\\
\frac{i+1-m}{m} & \mbox{if}\ i\geq m.
\end{array}\right.
\]
This shows that the desired diophantine exponent $\beta$ is equal to $\max\{0,\frac{n+1-m}{m}\}$.
\end{example}

\begin{example}[Weak non-planarity and extremality]\label{BKM-ctex}
Consider the unweighted case (i.e. $\alpha_i=\alpha_i'=1$). A submanifold $\cM \subset \Hom(V,E)$ is called \emph{weakly non-planar} if it is not contained in any proper pencil $\mathcal{P}_{W,r} \subsetneq \Hom(V,E)$. This notion was introduced, using slightly different words, by Beresnevich, Kleinbock and Margulis \cite{beresnevich-kleinbock-margulis} who showed that every locally good weakly non-planar measure on $\Hom(V,E)$ is extremal. The converse however does not hold, and we now provide an example.

Let $k\geq 4$ be an integer and $X=(\R^3)^k$. Consider the finite-dimensional space $V$ of polynomial maps $f$ from $X$ to $E=\R^3$ given by
\[ f(u_1,\dots,u_k) = \sum_{1\leq i<j\leq k} a_{ij} u_i\wedge u_j \]
where $\wedge$ is the usual wedge product in $\R^3$ and $a_{ij}\in\R$.
Let $\cM$ be the Zariski closure in $\Hom(V,E)$ of the image of $X$ by the map $\Phi:x\mapsto (f\mapsto f(x))$.

Let $W$ be the subspace of $V$ generated by the $u_1\wedge u_j$, $j=2,\dots,k$.
For any $x=(u_1,\dots,u_k)$ with $u_1\neq 0$, the space $W(x)$ is included in the orthogonal of $u_1$ and hence has dimension at most 2.
Therefore, $\cM\subset\cP_{W,2}$ is not weakly non-planar.

However, it is easy to see that $\GL_k(\R)$ acts irreducibly on $V$ by substitution of the variables. By the submodularity lemma we conclude that $V$ is the unique subspace realizing the maximum in $\tau(\cM)$. Therefore $\cM$ is not contained in any constraining pencil and thus $\cM$ must be extremal by Corollary~\ref{corextremal}.
\end{example}

\begin{example}[A criterion of Kleinbock-Margulis-Wang]

Here we recover by means of Corollary~\ref{corextremal} a criterion due to Kleinbock, Margulis and Wang \cite[Theorem~7.1(b)]{kleinbock-margulis-wang} for the extremality of measures on two-by-two matrices. Following their terminology, we say that an algebraic subset $\cM\subset M_{2,2}(\R)$ is \emph{row-nonplanar} if for all non-zero $v\in\R^2$, the map $M\mapsto Mv$ is not constant on $\cM$.

\begin{theorem}\label{criterion}
\cite[Theorem~7.1 (b)]{kleinbock-margulis-wang}
Let $\cM$ be an algebraic subset in $M_{2,2}(\R)$ and assume that both $\cM$ and its image $^t\cM$ under the transpose map are row-nonplanar.
Then the submanifold
\[ \widetilde{\cM}=\{ (I_2|M) \,;\, M\in\cM\} \subset M_{2,4}(\R)\]
is extremal.
\end{theorem}
\begin{proof}
 Corollary~\ref{corextremal}, all we have to do is to check that for all non-zero subspace $W<V=\R^4$, there exists a point $x\in\widetilde{\cM}$ such that $\dim x(W)\geq \frac{\dim W}{2}$.
We study all possible values for $\dim W$. Suppose first $\dim W = 1,2$. Since $\cM$ is row-nonplanar, there exists $x\in\widetilde{\cM}$ such that $x(W)$ is non-zero, which implies $\dim x(W)\geq 1 \geq \frac{\dim W}{2}$. Now if $\dim W = 3$, denote by $v_1,\dots,v_4$ the coordinates in $V$, and suppose first that $W$ is given by an equation $v_1=\lambda_2 v_2+\lambda_3 v_3+\lambda_4 v_4$.
The matrix of the restriction of $(I_2|M)$ to $W$ is given in some basis by
\[ \left( \begin{matrix}
\lambda_2 & a+\lambda_3 & b+\lambda_4\\
1 & c & d
\end{matrix}\right), \]
where $M=\left(\begin{matrix} a & b\\ c & d\end{matrix}\right)$. We have to show that all three columns are not always proportional. The minors corresponding to the pairs of columns $(1,2)$ and $(1,3)$ are $d_{12}=c\lambda_2-a-\lambda_3$ and $d_{13}=d\lambda_2-b-\lambda_4$, i.e.
\[ \left(\begin{matrix} d_{12}\\ d_{13}\end{matrix}\right) =
\left(\begin{matrix} a & c\\ b & d\end{matrix}\right)\left(\begin{matrix} -1\\\lambda_2\end{matrix}\right) - \left(\begin{matrix}\lambda_3\\ \lambda_4\end{matrix}\right). \]
Since the image of $\cM$ under the transposition is row-nonplanar, $d_{12}$ and $d_{13}$ cannot both vanish identically on $\cM$, so that there exists $x\in\widetilde{\cM}$ with $\dim x(W)=2$.
The remaining cases, where $W$ is defined by $v_2=\lambda_3 v_3+\lambda_4 v_4$, $v_3=\lambda_4 v_4$ or $v_4=0$ are treated similarly.
\end{proof}

\end{example}

\begin{example}[Every submanifold is extremal for a random lattice] The upper bound in Theorem~\ref{general-inequality} can be strict. Of course, if the maximum in $\tau(\cM)$ is attained on a $\Delta$-rational subspace, then this upper bound is the true value of $\widehat{\beta}(\cM)$, but otherwise it may not be, and we give here a family of examples illustrating this fact.
Note that the definition of $\tau(\cM)$ is independent of the lattice $\Delta$, whereas it is likely that varying the lattice will change the value of $\widehat{\beta}(\cM)$. We prove:

\begin{proposition}[Random lattice]\label{randomlattice}
Let $\cM$ be an irreducible algebraic set in $\Hom(V,E)$, and denote by $m$ the maximal rank of an element of $\cM$.
Suppose we pick the lattice $\Delta<V$ at random according to Haar measure on the space of (say unimodular) lattices in $V$.
Then $\widehat{\beta}(\cM)=\frac{\dim V-m}{m}$ with respect to almost every lattice $\Delta$.
\end{proposition}
\begin{proof}
Let $\beta_\Delta(x)$ be the number defined in $(\ref{diophexpdef})$.
Fix a lattice $\Delta_0$ in $V$ once and for all.
By Fubini's theorem, it is enough to prove that for almost every $x \in\cM$, $\beta_{g\Delta_0}(x)=\frac{\dim V-m}{m}$ for almost every $g\in {\rm SL}(V)$. Indeed, the set of pairs $(x,g) \in \cM \times {\rm SL}(V)$, for which $\beta_{g\Delta_0}(x)=\frac{\dim V-m}{m}$ is a Borel set and Fubini's theorem applied to the product measure $\nu \otimes dg$, where $dg$ is a Haar measure on ${\rm SL}(V)$ then implies that for almost every $\Delta$ we have $\beta_\Delta(x)=\frac{\dim V-m}{m}$ for almost every $x\in\cM$.

Fix $x\in\cM$ such that $\rk x=m$ (this happens for $x$ in a Zariski dense open subset of $\cM$ and hence for almost every $x$ in $\cM$.)
Consider the map from ${\rm SL}(V)$ to $\Hom(V,E)$ given by $\phi_x:g\mapsto xg^{-1}$.
The pushforward $\nu_x$ of the Haar measure on a neighborhood $U$ of some $g_0$ by the map $\phi_x$ is locally good.
Let $\cM_x$ be the Zariski closure of the support of $\nu_x$.
It contains all points of the form $xg^{-1}$, $g\in {\rm SL}(V)$. Now suppose $W$ is a subspace of $V$.
Since $\rk x = m$, we may choose $g\in {\rm SL}(V)$ such that $\dim (xg^{-1})(W)=\min\{m,\dim W\}$.
This shows that $\tau(\cM_x)=\frac{\dim V}{m}$.
Moreover, equality is attained for $W=V$, which is $\Delta_0$-rational, so Theorem~\ref{uppboundnudrat} shows that for almost every $y$ in $\cM_x$, $\beta_{\Delta_0}(y)=\frac{\dim V-m}{m}$. In particular, for almost every $g$ in $U$,
\[ \beta_{g\Delta_0}(x) = \beta_{\Delta_0}(xg^{-1}) = \frac{\dim V-m}{m}.\]
\end{proof}

\begin{remark}  The above proposition implies that for almost every $W$ in the Grassmannian of $k$-planes in $V$, the pencil $\cP_{W,r}$ is extremal, even if $\frac{\dim W}{r}>\frac{\dim V}{\dim E}$, provided it contains a matrix of maximal rank (i.e. provided $\dim W-r\leq \dim V-\dim E$). Indeed, $\widehat{\beta}_{g\Delta_0}(\cP_{W,r})=\widehat{\beta}_{\Delta_0}(\cP_{gW,r})$.
\end{remark}

\end{example}

\section{The critical exponent for rational nilpotent Lie groups}
\label{section:critical}

In this section we shall prove Theorems \ref{existence} and \ref{rationality-stability}.
Our method is based on the results of the previous sections regarding diophantine approximation on submanifolds of matrices.

Let $G$ denote an arbitrary simply connected nilpotent real Lie group, with Lie algebra $\g$ of nilpotency class $s$.
We endow $G$ with a left-invariant geodesic metric $d(\cdot,\cdot)$.
It is well known \cite{berestowski} that these are exactly the left-invariant Carnot-Carathéodory-Finsler metrics on $G$.
These metrics are obtained by the same construction as the left-invariant Riemannian metrics on $G$, except that instead of a Euclidean norm on $\g$, we start with an arbitrary norm on a generating subspace $V^1$ of $\g$.
Denoting inductively, for $i\geq 1$, $V^{i+1}=V^i+[V^1,V^i]$, every  Carnot-Carathéodory-Finsler metric on $G$ associated to $V^1$ is comparable near the identity (up to multiplicative constants) to
\begin{equation}\label{CCF}
d(\exp(X),1) \simeq |X|':=\max_{i=1,\ldots,s} \dist(X,V^{i-1})^{\frac{1}{i}},
\end{equation}
where $V^0=0$ by convention, $\dist$ is some fixed Euclidean distance on the Lie algebra, $s$ is the nilpotency class of $\g$ and $X \in \g$ lies in a neighborhood of the origin.
Hence the function  $X \mapsto d(\exp(X),1)$ is a local quasi-norm (see Definition~\ref{quasinorm}).

\subsection{Growth of a generic subgroup and group of words maps}

Let $k \geq 1$ and a $k$-tuple $\bg=(g_1,\ldots,g_k)$ of elements of $G$.
Recall that the subgroup $\Gamma_{\bg}$ generated by $\bg$ is $\beta$-diophantine if there exists $c>0$ such that, for every integer $n$,
\[ \min_{x\in S^n\setminus\{1\}} d(x,1) \geq c\cdot|S^n|^{-\beta},\]
where $S=\{1,g_1^{\pm 1},\dots,g_k^{\pm}\}$ and $S^n$ is the set of elements that can be obtained as a word of length $n$ in the elements of $S$.
This definition involves the volume $V_{\bg}(n)=|S^n|$ of the ball of radius $n$ in the subgroup generated by $S=\{1,g_1^{\pm 1},\ldots,g_k^{\pm 1}\}$.

It turns out \cite[Lemma~2.5]{abrs} that for a generic $k$-tuple $\bg$ (generic with respect to Haar measure on $G^k$) the isomorphism class of the subgroup $\Gamma_{\bg} = \langle S \rangle$ is always the same. It is isomorphic to the \emph{group of word maps on $k$ letters} $F_{k,G}$ of $G$, introduced in \cite{abrs}. We briefly recall this notion. Any element $w$ in the free group $F_k$ on $k$ letters $x_1,\dots,x_k$ determines a word map $w_G:G^k\to G$ given by replacing each letter $x_i$ by an element of $G$. The group $F_{k,G}$ is defined to be the set of all such word maps, with composition law given by $w_Gw_G'=(ww')_G$, where $ww'$ denotes the product of $w$ and $w'$ in $F_k$. It can also be viewed as the relatively free group in the variety of $k$-generated subgroups of $G$.

Since $F_{k,G}$ is a fixed nilpotent group, the Bass-Guivarc'h formula \cite{bassformula,guivarchformula} tells us that, up to multiplicative constants, for a generic $k$-tuple $\bg$ in $G$,
\begin{equation}\label{growthexp}V_{\bg}(n) \simeq n^{\eta_G(k)}\end{equation}
where $\eta_G(k)$ is a positive integer that can be expressed in terms of the ranks of the successive quotients of the central descending series of $F_{k,G}$, see $(\ref{bgfor})$ below.

Now, given a $k$-tuple $\bg$ of elements of $G$, we can define another diophantine exponent, denoted $\alpha(\bg)$, as follows. It is the infimum of all $\alpha>0$ such that for all but finitely many $\omega\in F_{k,G}$ we have:
\begin{equation}\label{alephdioph} d(\omega(\bg),1) \geq \ell(\omega)^{-\alpha},\end{equation}
where $\ell(\omega)$ is the length of an element $\omega\in F_{k,G}$, i.e. the minimal length of a word $w$ such that $w_G=\omega$.
Recall that the diophantine exponent of the subgroup $\Gamma_{\bg}$ was defined in Section~\ref{section:ergodicity} by
\[ \beta(\Gamma_{\bg})=\inf\{\beta>0\ |\ \Gamma_{\bg}\ \mbox{is}\ \beta\mbox{-diophantine}\}.\]
It follows from the above observations about $V_\bg(n)$  that, for almost every $k$-tuple $\bg$ in $G$,
\begin{equation}\label{alephbet} \alpha(\bg) =  \eta_G(k) \beta(\Gamma_{\bg}).\end{equation}

\subsection{From words to Lie brackets}

In this subsection, we describe the correspondence between words on $G$ and laws on its Lie algebra $\g$.

We first define the Lie algebra $\cF_{k,\g}$ of bracket maps on $k$ letters on $\g$.
Let $\cF_k$ be the free Lie algebra on $k$ generators. Each element $\sr$ in $\cF_k$ yields a map
$$
\begin{array}{cccc}
\sr : & \g^k & \rightarrow & \g \\
& (X_1,\dots,X_k) & \mapsto & \sr(X_1,\dots,X_k)
\end{array}
$$
where $\sr(X_1,\dots,X_k)$ is the evaluation of the formal bracket $\sr$ at the point $(X_1,\dots,X_k)$ in $\g^k$.
By definition, the Lie algebra $\cF_{k,\g}$ consists of all maps from $\g^k$ to $\g$ obtained in the above manner.
It is naturally isomorphic to the quotient Lie algebra $\cF_k/\cL_{k,\g}$ where $\cL_{k,\g}$ is the \emph{ideal of laws} on $k$ letters on the Lie algebra $\g$. Recall that a law on $k$ letters is an element $\sr \in \cF_k$ such that $\sr(X_1,\dots,X_k)=0$ for all $X_1,\ldots,X_k$ in $\g$.

Note that the free Lie algebra $\cF_k$ has a natural $\Q$-structure induced by the subring $\cF_k(\Z)$ of integer linear combinations of brackets monomials. We can thus consider the \emph{ideal of rational laws} $\cL_{k,\g,\Q}$, which is the real span of the intersection of $\cL_{k,\g}$ with $\cF_k(\Z)$. It thus inherits a rational structure, which induces a rational structure on the quotient space $\cF_{k,\g,\Q}= \cF_k /\cL_{k,\g,\Q}$. We will say that $\sr$ is an element of $\cF_{k,\g,\Q}(\Z)$ if it is the projection on $\cF_{k,\g,\Q}$ of an element of $\cF_k(\Z)$.

\begin{remark}
Certainly if $\g$ itself is defined over the rationals, then so is $\cL_{k,\g}$, but the converse does not hold.
For example, it followed from the analysis made in \cite[Appendix A]{abrs} that the ideal of laws is always defined over $\Q$ if $\g$ is a nilpotent Lie algebra of step at most 5, or if $\g$ is both nilpotent and metabelian.
\end{remark}

The Lie algebra $\cF_{k,\g\,\Q}$ has a graded structure
$$\cF_{k,\g,\Q} = \bigoplus_{i=1}^s \cF_{k,\g,\Q}^{[i]},$$
where $\cF_{k,\g,\Q}^{[i]}$ is the  homogeneous part of $\cF_{k,\g,\Q}$ consisting of brackets of degree $i$.
For $\sr=\sum \sr_i$ with $\sr_i\in \cF_{k,\g,\Q}^{[i]}$, we let \begin{equation}\label{homonorm}|\sr|=\max_{i=1,\ldots,s} \|\sr_i\|^{\frac{1}{i}},\end{equation} where $\|\cdot\|$ is a fixed norm on $\cF_{k,\g,\Q}$.

\begin{lemma}\label{words-brackets}
Let $\g$ be a real nilpotent Lie algebra and $G$ the simply connected Lie group with Lie algebra $\g$.
There are positive integers $C,D$ such that
\begin{itemize}
\item If $\omega\in F_{k,G}$, then there exists $\sr\in \cF_{k,\g,\Q}(\Z)$ with $|\sr|\leq D\ell(\omega)$ such that for all $X_1,\dots,X_k$ in $\g$,
$\omega(e^{X_1},\dots,e^{X_k})=e^{\frac{1}{C}\sr(X_1,\dots,X_k)}$.
\item If $\sr\in \cF_{k,\g,\Q}(\Z)$, then there exists $\omega\in F_{k,G}$ with $\ell(\omega)\leq D|\sr|$ and for all $X_1,\dots,X_k$ in $\g$,
$e^{C\sr(X_1,\dots,X_k)}=\omega(e^{X_1},\dots,e^{X_k})$.
\end{itemize}
\end{lemma}

\begin{proof} This was proved in \cite[Lemma~3.5]{abrs} for the free Lie algebra $\cF_{k}$ and the free group $F_k$. The relative version stated here follows without difficulty, using that $\cF_{k,\g,\Q}=\cF_{k}/\cL_{k,\g,\Q}=\oplus_i \cF_{k}^{[i]}/\cL^{[i]}_{k,\g,\Q}$.
\end{proof}

Recall that $|X|'$ denotes the local quasi-norm on $\g$ defined in $(\ref{CCF})$. The above lemma has the following immediate consequence.

\begin{proposition}
Let $G$ be a simply connected nilpotent Lie group, with Lie algebra $\g$. Let $\bg=(e^{X_1},\dots,e^{X_k})$ be a $k$-tuple in $G$. The exponent $\alpha(\bg)$ defined in $(\ref{alephdioph})$ is also the infimum of all $\alpha>0$ such that
\begin{equation}\label{liediophantine}
|\sr(X_1,\dots,X_k)|' \geq |\sr|^{-\alpha}
\end{equation}
holds for all but finitely many $\sr\in\cF_{k,\g,\Q}(\Z)$.
\end{proposition}

It also follows from Lemma \ref{words-brackets} that the group of word maps $F_{k,G}$ is included as a lattice in the simply connected rational nilpotent Lie group whose Lie algebra is $\cF_{k,\g,\Q}$ \cite[Proposition 3.9]{abrs}.
The Bass-Guivarc'h formula for the growth exponent of $F_{k,G}$ from $(\ref{growthexp})$ now reads:
\begin{equation}\label{bgfor}
\eta_G(k) = \sum_{i=1}^s i \dim \cF_{k,\g,\Q}^{[i]}.
\end{equation}
Each $\cF_{k,\g,\Q}^{[i]}$ is a module for the action of $\GL_k$ under linear substitution. We will see below that, as a consequence, $\dim \cF_{k,\g,\Q}^{[i]}$ is a degree $i$ polynomial in $k$ with rational coefficients, when $k$ is large enough. It follows that, for $k$ large, $\eta_G(k)$ is given by a degree $s$ polynomial in $k$.
Therefore, in view of $(\ref{alephbet})$, we now focus on computing $\alpha(\bg)$ for a random tuple $\bg$.

\subsection{Existence of the exponent}\label{ratsec}
We now use the results of Section~\ref{quasinorm-sec} to study the diophantine problem described by (\ref{liediophantine}), and prove Theorem~\ref{existence} from the introduction, which we recall here:

\begin{theorem}[Existence of the exponent]
\label{existence2}
Let $G$ be a connected and simply connected nilpotent real Lie group endowed with a left-invariant geodesic metric $d$.
For each $k\geq 1$, there is $\widehat{\beta}_k \in [0,+\infty]$ such that for almost every $k$-tuple $\bg \in G^k$ with respect to Haar measure, we have
$$\beta(\Gamma_{\bg}) = \widehat{\beta}_k.$$
\end{theorem}
\begin{proof}
We set $E=\g$, $V=\mathcal{F}_{k,\g,\Q}$, $\Delta=\mathcal{F}_{k,\g,\Q}(\Z)$ and we define the local quasi-norm $|X|'$ on $E=\g$ by $(\ref{CCF})$ and the quasi-norm $|\sr|$ on $V$ by $(\ref{homonorm})$.
We let $U=\g^k$ and 
\begin{align*}
\Phi : &U \to \Hom(V,E)\\ &\bg \mapsto (\Phi(\bg) \colon \sr \mapsto \sr(\bg)).
\end{align*}
Now we are in the setting of Corollary \ref{analytic-existence} from Section \ref{quasinorm-sec} and hence $\alpha(\bg)$ is almost everywhere constant. In view of $(\ref{alephbet})$ this shows that $\beta(\bg)$ is also well-defined and constant almost everywhere.
\end{proof}

We denote by $\widehat{\alpha}_k$ the almost sure value of $\alpha(\bg)$ for $\bg \in \g^k$. Hence (\ref{alephbet}) becomes:

\begin{equation}
\label{alephbet2} \widehat{\alpha}_k =  \eta_G(k) \widehat{\beta}_k.
\end{equation}
If we assume moreover that $G$ and the geodesic metric are rational, we may even conclude that $\alpha_k$ is rational.
We define a \emph{rational} nilpotent Lie group as a nilpotent Lie group $G$ with Lie algebra $\g$ endowed with a rational structure induced by a basis $\cB$ with rational structure constants.
A left-invariant geodesic metric on a rational nilpotent Lie group $G$ is said to be \emph{rational} if the associated subspaces $V^i$ -- defined at the beginning of this section -- are rational.
For this, it is enough to require that $V^1$ is rational.

\begin{theorem}[Rationality of the exponent]
Let $G$ be a connected and simply connected rational nilpotent real Lie group, and $d(\cdot,\cdot)$ a rational left-invariant geodesic metric on $G$.
Then, for all $k\geq 1$,
\[ \widehat{\beta}_k\in\Q.\]
\end{theorem}
\begin{proof}
We use the notation of the previous proof.
If $\g$ is rational, then the Zariski closure $\mathcal{M}$ of $\Phi(U)$ is defined over $\Q$.
Indeed, complexifying $E$, $V$ and $\Phi$ we see that $\sigma(\Phi(\bg)) = \Phi(\sigma(\bg))$, for all $\sigma \in \Gal(\C|\Q)$.
Therefore, Theorem~\ref{rathm} applies and we get that $\widehat{\alpha}_k = \tau_\Q(\mathcal{M})$.
Now note that the exponents $\alpha_i$ and $\alpha'_i$ defining our quasi-norms are integers.
As a result, the functions $\psi$ and $\phi$ are integer-valued, and thus $\widehat{\alpha}_k \in \Q$. 
\end{proof}

\subsection{The relatively free Lie algebra as a $\GL_k$-module}

In order to prove the second part of Theorem \ref{rationality-stability}, we need some understanding of the $\GL_k$-submodules of $\cF_{k,\g,\Q}$.
The action of the linear group $\GL_k$ on the free Lie algebra $\cF_k$ on $k$ letters is by substitution of the variables.
It turns out that the decomposition of each homogeneous component $\cF_k^{[i]}$ into simple $\GL_k$-modules is representation stable in the sense of Church, Ellenberg and Farb, see \cite[Corollary 5.7]{church-farb}.
This can be made more precise as follows. We set 

$$\cF_k^{\leq s}  = \bigoplus_1^s \cF_k^{[i]}$$
the subspace of the free Lie algebra on $k$ letters spanned by brackets of order at most $s$.

\begin{proposition}\label{correspond}
Let $k \geq s$.
The forgetful map
\begin{align*} F_s \colon \cF_k^{\leq s} &\to \cF_s^{\leq s} \\
\sr(\sx_1,\ldots,\sx_k) & \mapsto\sr(\sx_1,\ldots,\sx_s,0,\ldots,0)
\end{align*}
establishes a bijective correspondence between $\GL_k$-submodules of $\cF_{k}^{\leq s}$ and $\GL_{s}$-submodules of $\cF^{\leq s}_{s}$. Moreover, if $W\leq\cF_{k}^{\leq s}$ is a $\GL_k$-submodule with the following decomposition into irreducible submodules
\begin{equation}\label{decdecW}W=\bigoplus_{\lambda} E^{\lambda}(k)^{n_\lambda},\end{equation}
then $F_s(W)$ decomposes, with the same multiplicities, as
$$F_s(W)=\bigoplus_{\lambda} E^{\lambda}(s)^{n_\lambda},$$
where $E^{\lambda}(k)$ is the irreducible representation of $\GL_k$ with Young diagram $\lambda$. 
\end{proposition}

\begin{remark}
The Young diagrams appearing in either decomposition have at most $s$ boxes. Furthermore $F_s$ sends $\cL_{k,\g}$ to $\cL_{s,\g}$ and $\cL_{k,\g,\Q}$ to $\cL_{s,\g,\Q}$, if $\g$ is any Lie algebra of nilpotency class $s$. In particular, as $k$ grows,  $\cF_{k,\g,\Q}$ has only boundedly many irreducible $\GL_k$-submodules counting multiplicity in its decomposition, all obtained by the above process from the decomposition of $\cF_{s,\g,\Q}$ into irreducible $\GL_s$-submodules.
\end{remark}

\begin{proof}
Note that elements of $\cF_{k}^{\leq s}$ are linear combinations of brackets having at most $s$ letters, and that $\cF_{k}^{\leq s}$ decomposes into weight spaces for the diagonal action
\[ (t_1,\ldots,t_k) \cdot c(\sx_1,\ldots,\sx_k)=c(t_1\sx_1,\ldots,t_k\sx_k).\]
If $W \leq \cF_{k}^{(s)}$ is an irreducible $\GL_k$-module, then it is generated by a highest weight vector, whose weight $\lambda$ is of the form $(n_1,\ldots,n_k)$ with $n_i=0$ if $i>s$. The corresponding $\GL_s$-submodule $F_s(W)$ is generated by the same highest weight vector.
It is therefore an irreducible $\GL_s$-module with the same Young diagram. The result follows.
\end{proof}


Given a Young diagram $\lambda$, by the Weyl dimension formula,  the dimension $d_\lambda(k)$ of the irreducible $\GL_k$-module associated to $\lambda$ is:
$$d_\lambda(k)=\prod_{1\leq i<j\leq k} \frac{\lambda_i-\lambda_j+j-i}{j-i}.$$
In particular, if $i$ is the total number of boxes of $\lambda$, then $d_\lambda(k)$ is a degree $i$ polynomial in $k$ with rational coefficients. The number of boxes of $E^\lambda \leq \cF_{k}^{\leq s}$ is the number of letters appearing in the brackets of the associated submodule of $\cF_{k}^{\leq s}$. In particular, we obtain:

\begin{corollary}\label{dimcor} For $i=1,\ldots,s$, the map $k \mapsto \dim \cF_{k,\g,\Q}^{[i]}$ is a polynomial in $k$ of degree $i$ with rational coefficients, provided $k\geq s$.
From $(\ref{bgfor})$ the same holds for the growth exponent $\eta_G(k)$.
More generally, there is a finite family $F$ of rational polynomials of degree at most $s$ such that for $k\geq s$, if $W$ is any $\GL_k$-submodule in $\cF_{k,\g,\Q}$ then $\dim W = P(k)$ for some $P\in F$.
\end{corollary}


\subsection{Stability of the exponent}

We now prove the second part of Theorem~\ref{rationality-stability} from the introduction, which we recall here, for convenience.

\begin{theorem}[Stability of the exponent]
\label{stability}
Let $G$ be a connected and simply connected rational nilpotent real Lie group, and $d(\cdot,\cdot)$ a rational left-invariant geodesic metric on $G$.
There exists a rational function $F_{\g}\in\Q(X)$ with rational coefficients such that, for $k$ large enough,
\[ \widehat{\beta}_k = F_{\g}(k).\]
\end{theorem}
\begin{proof}
We already know from Corollary \ref{dimcor} that, for $k\geq s$, $\eta_G(k)$ is given by a degree $s$ polynomial in $k$ with rational coefficients.
So in view of $(\ref{alephbet2})$, we need only prove that $\widehat{\alpha}_k$ is given by a rational function, for large $k$.
We saw in \S\ref{ratsec} that $\widehat{\alpha}_k= \tau_\Q(\mathcal{M})$, where $\tau_\Q(\mathcal{M})$ is given by $(\ref{tauq})$ and takes the form $(\ref{tausub})$ of the ratio of submodular functions, $-\psi_\mathcal{M}$ and $\phi_\mathcal{M}$ defined on the grassmannian $\Grass(V)$, where $V=\cF_{k,\g,\Q}$.

The group $\GL_k$ acts by linear substitution on $V$ and it preserves the flag of subspaces
\begin{equation}\label{vicentral}
V_i=\oplus_{j\geq i} \cF_{k,\g,\Q}^{[j]} 
\end{equation}
defining the quasi-norm $|\sr|$ in $(\ref{homonorm})$, so the function $\psi$ is $\GL_k$-invariant.
Noting that $gW \cap \ker \Phi(u) = g( W \cap \ker \Phi(g^{-1}u))$ for any $g \in \GL_k$, $u \in \g^k$ and $W \leq V$, we conclude that $\psi_\mathcal{M}$ is $\GL_k$-invariant.
Similarly $\phi_\mathcal{M}$ is $\GL_k$-invariant.
We can therefore apply the submodularity Lemma \ref{submodular}, and conclude that there is a $\GL_k$-invariant $\Delta$-rational subspace $W\leq V$ such that \begin{equation}\label{alphamax}\widehat{\alpha}_k = \frac{\psi_\mathcal{M}(W)}{\phi_\mathcal{M}(W)}.\end{equation}

Combining $(\ref{homonorm})$ with $(\ref{psidec})$ and $(\ref{CCF})$ with $(\ref{phidec})$ we obtain $\psi_\mathcal{M}$ and $\phi_\mathcal{M}$ explicitly in the form:
\begin{align}\label{psiexp} 
&\psi_\mathcal{M}(W) = \sum_{i=1,\ldots,s}\big( \dim(W\cap V_i) - \dim((W \cap V_i)(\sx))\big) \\
&\phi_\mathcal{M}(W) = \sum_{i=0,\ldots,s-1} \dim \pi_i(W(\sx))\label{phiexp}
\end{align}
for all $\sx \in \g^k$ in a Zariski open subset, where $\pi_i: \g \to \g/V^i$. 
Recall that $V_i\leq V=\cF_{k,\g,\Q}$ is defined in (\ref{vicentral}) above, and $V^i\leq\g$ at the beginning of this section; also note that this notation is not perfectly coherent with the notation in Section~\ref{quasinorm-sec}, where $\{V_i\}$ and $\{V_i'\}$ are \emph{full} flags.

According to Proposition \ref{correspond} and Corollary \ref{dimcor} the dimension of $\GL_k$-invariant subspaces of $V=\cF_{k,\g,\Q}$ can achieve only boundedly many values, each given by a polynomial in $k$ with rational coefficients and degree at most $s$. It then follows from the special form (\ref{psidec}) taken by $\psi$ that $\psi_\mathcal{M}$ can only take boundedly many values, each of which is one of boundedly many polynomials in $k$ of degree at most $s$ and with rational coefficients. We conclude that the same holds for $\widehat{\alpha}_k$, since $\phi$ is integer valued and bounded in terms of $\dim \g$ only.
When $k$ is large enough (larger than a constant depending only on the size of the coefficients of these polynomials, hence only on $\dim \g$), this maximum will be achieved by a single polynomial with rational coefficients and of degree at most $s$. In fact the degree will be exactly $s$, because of the Dirichlet lower bound and the fact that already with $W=\cF_{k,\g,\Q}^{[s]}$, the function $k\mapsto\frac{\psi_\mathcal{M}(W)}{\phi_\mathcal{M}(W)}$ is a degree $s$ polynomial.
This completes the proof.
\end{proof}

Now that we have finished proving Theorem~\ref{rationality-stability}, we briefly explain how to derive Corollary~\ref{theorem-limit}.

\begin{proof}[Proof of Corollary \ref{theorem-limit}] The existence of the limit follows immediately from the theorem.
To prove the upper bound, note that $\phi$ takes integer values, while $\psi$ is bounded above by $s \dim \cF_{k,\g,\Q}$, as follows from $(\ref{psidec})$.
However, according to the Bass-Guivarc'h formula (\ref{bgfor}), $\eta_G(k)$ is asymptotic to $s\dim \cF_{k,\g,\Q}$.
This shows the upper bound.
For the lower bound, note that $\tau_\Q(\mathcal{M})$ is a maximum over the subspaces $W$; evaluating at $W=\cF_{k,\g,\Q}^{[s]}$, we get $\psi_\mathcal{M}(W)= s (\dim \cF_{k,\g,\Q}^{[s]}-\dim \g^{[s]})$ and $\phi_{\cM}(W)\leq s \dim \g^{[s]}$.

When $d(\cdot,\cdot)$ is Riemannian, all $\alpha_i'$ are $1$ and then $\phi(W) \leq \dim \g^{[s]}$. This ends the proof. We will see explicit examples in the next section, where both upper and lower bounds are attained for $\lim_k \widehat{\beta}_k$, see also Remark \ref{strict}.
\end{proof}

\begin{remark}[Irrational $\g$]\label{irrational} If $\cL_{k,\g}$ is not rational, then there is a non-trivial subspace $W \leq \cF_{k,\g,\Q}$ for which $\phi_\mathcal{M}(W)=0$, and in particular $\tau(\mathcal{M}) = + \infty$, so we cannot conclude anything in this case. However if $\cL_{k,\g}$ is rational and even if $\g$ is not, we can assert that the conclusion of Theorem \ref{rationality-stability} holds in certain cases. For example it holds whenever $\cF_{k,\g}$ is multiplicity-free as a $\GL_k$-module. Indeed in this case every $\GL_k$-invariant subspace is rational (because $\GL_k$ is $\Q$-split). Since we know that the maximum in $\tau(\mathcal{M})$ is attained at a $\GL_k$-invariant subspace, we obtain again $\tau(\mathcal{M})=\tau_\Q(\mathcal{M})$ in this case. We will make use of this observation in some of the examples below.
\end{remark}

\begin{remark}[The exponent is attained in the last step]\label{lastep}
Let
\[ L=\{\sr \in \cF_{k,\g,\Q} \ |\ \sr(\g^k) \subset \g^{(s)}\}.\]
When $k$ is large enough, the $\GL_k$-invariant rational subspace $W$ realizing the maximum of $\psi_\mathcal{M}/\phi_\mathcal{M}$ in $(\ref{alphamax})$ can be chosen to belong to $L$. This can be seen easily by writing $W=W_0\oplus W_0'$ with $W_0 =W \cap L$ for a $\GL_k$-invariant $W_0'$ (which exists by complete reducibility of the $\GL_k$ action). Then from $(\ref{psiexp})$ it follows that $\psi_\mathcal{M}(W) - \psi_\mathcal{M}(W_0)$ is a polynomial of degree at most $s-1$ in $k$, while $\phi_\mathcal{M}(W) > \phi_\mathcal{M}(W_0)$ unless $W=W_0$.
From these two inequalities we see that, for large enough $k$,
$$\frac{\psi_\mathcal{M}(W)}{\phi_\mathcal{M}(W)} \leq \frac{\psi_\mathcal{M}(W_0)}{\phi_\mathcal{M}(W_0)}.$$
\end{remark}

\section{Explicit values of the critical exponent in some examples}\label{section:examples}

In this final section, we illustrate Theorems \ref{rationality-stability} and \ref{theorem-limit} and work out an explicit value for the critical exponent $\widehat{\beta}_k$ in several examples. For definiteness we always assume in this section that the metric on $G$ is Riemannian.

\subsection{Nilpotent Lie groups of step 2}

In this paragraph, $G$ is a connected simply connected non-abelian nilpotent Lie group of step 2.
We denote by $d_i$, $i=1,2$ the dimension of $\g^{[i]}$. Here $\g^{[1]}=\g/[\g,\g]$ and $\g^{[2]}=[\g,\g]$.

It is easy to check that $d_2 \leq \frac{d_1(d_1-1)}{2}$. Moreover, by \cite[Appendix~A]{abrs}, $\cF_{k,\g}^{[1]}$ and $\cF_{k,\g}^{[2]}$ are both irreducible $\GL_k$-modules of dimension $k$ and $k(k-1)/2$ respectively. In particular Remark \ref{irrational} applies and this implies that $\widehat{\alpha}_k=\frac{k(k-1)}{d_2}-2.$ Since $\eta_G(k)=k+k(k-1)=k^2$ according to $(\ref{bgfor})$, we get:

\begin{theorem}[Step 2 nilpotent Lie groups]
Let $G$ be a connected simply connected non-abelian nilpotent Lie group of step 2. Set $d_1=\dim G/[G,G]$ and $d_2=\dim [G,G]$ and let $k\geq d_1$ be an integer. The critical exponent for $k$-tuples in $G$ is
$$\widehat{\beta}_k=\frac{1}{d_2} -\frac{1}{d_2 k}- \frac{2}{k^2}.$$
\end{theorem}

In the special case of the $3$-dimensional Heisenberg group, $d_2=1$ and we thus recover the computation made in Section \ref{section:heisenberg}.

\subsection{Metabelian nilpotent Lie groups}

Suppose now, more generally, that $G$ is a simply connected metabelian nilpotent Lie group, that is we assume that $[G,G]$ is abelian.  This does not constrain the nilpotency class. No assumption of rationality on $G$ is made. It is shown in \cite[\S A.3]{abrs} that $\cF_{k,\g}^{[i]}$ is irreducible as a $\GL_k$-module for each $i$ and isomorphic to $E^{(i-1,1)}(k)$. Hence Remark \ref{irrational} applies.
In particular, for $k$ large we obtain:
$$\widehat{\alpha}_k = s\frac{\dim E^{(s-1,1)}(k)}{\dim G^{(s)}}-s.$$
On the other hand, the Bass-Guivarc'h formula reads:
$$\eta_G(k) = k+\sum_{i=2}^s i  \dim E^{(i-1,1)}(k).$$
Using Weyl's dimension formula, we may compute $\dim E^{(i-1,1)}(k) = (i-1){i+k-2 \choose i}$ a polynomial of degree $i$ in $k$.

\begin{theorem}[Metabelian nilpotent groups] When $k$ is large enough the critical exponent is given by $\widehat{\beta}_k=\widehat{\alpha}_k /\eta_G(k)$ with the above polynomial expressions for $\widehat{\alpha}_k$ and $\eta_G(k)$. In particular,
$$ \lim_{k\rightarrow\infty} \widehat{\beta}_k = \frac{1}{\dim G^{(s)}}.$$
\end{theorem}

\subsection{Unipotent upper triangular matrices}

We now deal with the case of the group $U_s$ of upper triangular unipotent $(s+1)\times (s+1)$ matrices. Its Lie algebra is the Lie algebra $\g=\uu_s$ of upper triangular matrices with zero diagonal, and our first task will be to determine, for any positive integer $k$, the Lie algebra $\cF_{k,\g}$ of bracket maps on $\g$ on $k$ letters. The result is the following.

\begin{proposition}\label{unipotentlaws}
Let $k$ be a positive integer, and let $\g=\uu_s$ be the Lie algebra of upper triangular $(s+1)\times(s+1)$ matrices with zero diagonal terms. The Lie algebra $\cF_{k,\g}$ of bracket maps on $\g$ on $k$ letters is isomorphic to the free $s$-step nilpotent Lie algebra $\cF_{k,s}$.
\end{proposition}

An alternative formulation of Proposition~\ref{unipotentlaws} is in terms of the nilpotent group of unipotent upper triangular matrices:

\begin{corollary}
Let $U_s$ be the group of upper triangular unipotent $(s+1)\times (s+1)$ matrices, and let $k$ be any positive integer. Then $U_s$ contains the free nilpotent group of step $s$ on $k$ generators as a finitely generated subgroup.
\end{corollary}

\begin{proof}[Proof of Proposition~\ref{unipotentlaws}]
Let $k\geq s$ be a positive integer and let  $\cF_k$ denote the free Lie algebra on the $k$ generators $\sx_1,\sx_2,\dots,\sx_k$. We have to check that $\g=\uu_s$ has no non-trivial relation of degree less than or equal to $s$ in $\cF_{k}$. First, by the proof of \cite[Theorem~3, page~99]{bahturin} we note that if $\g$ satisfies a non-trivial relation, then it also satisfies a non-trivial multilinear relation whose degree is not larger. Now, if $\sr$ is such a multilinear relation in $\g$, then so is each of its homogeneous components, so we may assume $\sr$ has degree one in each $\sx_i$, $1\leq i\leq t$. In fact, we may also assume $t=s$, otherwise, we replace $\sr$ by $[[\sr,\sx_{t+1}],\dots,\sx_{s}]$. So we just have to see that $\g$ has no multilinear relation in $s$ variables.

Let $\cH_s$ be the vector space of all elements of $\cF_k$ of degree $s$ that are multilinear in $(\sx_1,\sx_2,\dots,\sx_s)$. We want to check that the canonical map $\theta:\cH_s\rightarrow \cF_{k,\g}$ is injective. By Witt's Formula \cite[Theorem~11.2.2]{hall},
\begin{equation}\label{dims1}
\dim \cH_s= (s-1)!
\end{equation}
For $\sigma$ a permutation of $\{1,2,\dots,s\}$ fixing $1$, we denote
$$m_\sigma=[\dots[[\sx_1,\sx_{\sigma(2)}],\sx_{\sigma(3)}],\dots,\sx_{\sigma(s)}].$$
This gives us a family $(m_\sigma)$ of $(s-1)!$ elements in $\cH_s$. We will show that its image $(\theta(m_\sigma))$ under $\theta$ is linearly independent in $\cF_{k,\g}$; together with (\ref{dims1}), this will prove the theorem.

For $1\leq i,j\leq s+1$ we denote by $E_{i,j}$ the matrix whose only non-zero entry is $1$, in position $(i,j)$. For $1\leq i\leq s$, we also let $e_i=E_{i,i+1}$. Using the relations $[E_{i,j},E_{k,l}]=\delta_{jk}E_{il}$, one can compute the values of the $\theta(m_\sigma)$ on permutations of the $s$-tuple $(e_i)$, and get, for any two permutations $\sigma$ and $\tau$ of $\{1,2,\dots,s\}$, both fixing $1$:
$$
\theta(m_\sigma)(e_1,e_{\tau(2)},\dots,e_{\tau(s)})=\left\{
\begin{array}{cc}
E_{s+1,s+1} & \mbox{if}\ \sigma=\tau^{-1}\\
0 & \mbox{otherwise.}
\end{array}
\right.
$$
This implies in particular that the family $(\theta(m_\sigma))$ is linearly independent, so we are done.
\end{proof}

From Proposition~\ref{unipotentlaws}, it is not difficult to compute the critical exponent for the group $G=U_s$ of unipotent upper triangular matrices. Indeed, if $\sr \in \cF_k$ is a law of $\g/\g^{(s)}$, then $[\sx_{k+1},\sr]$ is a law of $\g$, hence has all its homogeneous components of degree at least $s+1$ according to Proposition~\ref{unipotentlaws}. Hence $\sr \in \cF_k^{[s]}$. Therefore Remark \ref{lastep} applies and in computing the maximum $\tau(\mathcal{M})$ we may restrict attention to subspaces $W$ in $\cF_k^{[s]}$.
For such $W$, $\psi_\mathcal{M}(W)=s(\dim W - \dim W(\g^k))$ and $\phi_\mathcal{M}(W)\leq 1$, so we get
$$\widehat{\alpha}_k=s(\dim \cF_{k}^{[s]} -1) = \sum_{d|s}\mu(d)k^{s/d} -s $$
where we used Witt's formula \cite[Theorem~11.2.2]{hall} in the last equality.
From the Bass-Guivarc'h formula (\ref{bgfor}) and Witt's formula again, we also compute
\begin{equation}\label{growthud}
\eta_G(k)=\sum_1^s i \dim \cF_{k}^{[i]} =  \sum_1^s  \sum_{d |i} \mu(d) k^{i/d} = \sum_{i\leq s} M(\frac{s}{i})k^i,
\end{equation}
where $M(x)=\sum_{n \leq x} \mu(n)$ is the Mertens function, and thus obtain:

\begin{theorem}[Critical exponent for $U_s$]
Let $U_s$ be the nilpotent group of unipotent upper triangular $(s+1)\times (s+1)$ matrices. Then for $k$ large, the critical exponent for $k$-tuples in $U_s$ is
$$\widehat{\beta}_k=\frac{\sum_{d|s}\mu(d)k^{s/d} -s}{\sum_{i\leq s} M(\frac{s}{i})k^i} $$
with $M(x)=\sum_{n \leq x} \mu(n)$ the Mertens function.
In particular,  $\lim_k \widehat{\beta}_k =1$.
\end{theorem}
Note that the limit also follows directly from Theorem~\ref{theorem-limit} since for $G=U_s$, we have $\dim G^{(s)} = 1$.

\subsection{Free nilpotent Lie algebras} In this paragraph $\g$ will denote the free nilpotent Lie algebra $\cF_{d,s}$ of step $s$ on $d$ generators. We will assume throughout that $d \geq s$.

Note that $\g$ has a natural structure of $\GL_{d}$-module. Given a Young diagram $\lambda$, we denote by $d_{\lambda}(d)$ the dimension of the irreducible $\GL_d$-representation $E^\lambda(d)$ associated to $\lambda$.

\begin{theorem}[Critical exponent for free nilpotent groups]\label{freegroup}
Let $G$ be the connected simply connected Lie group with Lie algebra $\g=\cF_{d,s}$. Assume that $d \geq s$. Then, if $k$ is large enough, the critical exponent for $k$-tuples in $G$ is
$$\widehat{\beta}_k=\frac{s}{{d+1 \choose s}}\cdot \frac{{k+1 \choose s}-{d+1 \choose s}}{\sum_{i\leq s} M(\frac{s}{i})k^i}$$
where $M(x)=\sum_{n \leq x} \mu(n)$ is the Mertens function.
\end{theorem}

Passing to the limit, we get:

\begin{corollary}[Limiting value]\label{limiting}
For $d\geq s$, we have the following limit for the critical exponent for $\g=\cF_{d,s}$:
$$\lim_{k \to +\infty} \widehat{\beta}_k = \frac{1}{(s-1)!{d+1 \choose s}}.$$
\end{corollary}

\begin{remark}\label{strict} This shows that the strict inequality $\frac{1}{\dim \g^{(s)}} < \lim_k \widehat{\beta}_k < 1$ can happen.
\end{remark}

We now pass to the proof of Theorem \ref{freegroup}. Our Lie algebra $\g$ has the following special property: for each $k\geq s$ the laws of $\g/\g^{(s)}$ are either laws of $\g$ or are contained in $\cF_{k,\g}^{[s]}$ (i.e. if $\sr\in \cF_{k,g}$ and $\sr(\g^k)\leq \g^{(s)}$, then $\sr \in  \cF_{k,\g}^{[s]}$). So it follows from Remark \ref{lastep} that the maximum value of $\tau_\Q(\mathcal{M})$ is attained at a rational $\GL_k$-invariant subspace $W$ that is contained in $\cF_{k,\g}^{[s]}$. For such a subspace $\psi_\mathcal{M}(W)=s(\dim W - \phi_\mathcal{M}(W))$ and $\phi_\mathcal{M}(W)= \dim W(\sx)$, where $W(\sx)$ is the image of $W$ in $\g$ under evaluation at a generic $\sx$ in $\g^k$.
Note that if $k\geq s$, then, for $\sx$ in a dense Zariski open set in $\g^k$, the space $W(\sx)$ is independent of $\sx$, equal to the span $W_{\g}$ of all $\sr(\sx_1,\ldots,\sx_k)$, $\sr \in W$, $\sx_1,\ldots,\sx_k \in \g$. 
Thus,
$$\frac{1}{s}\widehat{\alpha}_k = \max \{ \frac{\dim W}{\dim W_\g}-1 \,;\, W \leq \cF_{k,\g}^{[s]} \textnormal{ rational }\GL_k\textnormal{-module}\}.$$

Recall Proposition~\ref{correspond}, which describes precisely the $\GL_k$-submodules of $\cF_{k,\g}^{[s]}$. An immediate consequence of this proposition is the following simpler expression:
\begin{equation}\label{sigfor}\frac{1}{s}\widehat{\alpha}_k = \max_\lambda \left\{ \frac{d_\lambda(k)}{d_\lambda(d)}\right\}-1,\end{equation}
where the maximum is taken over all Young diagrams $\lambda$ that appear in the decomposition of $\g^{(s)}$ as a $\GL_d$-module. The following theorem of Klyachko describes this set of diagrams:

\begin{theorem}[Klyachko \cite{klyachko,reutenauer}]\label{klyachko}
Let $\g=\cF_{d,s}$. Then, except for $(s)$ and $(1,1,\dots,1)$, with the further exceptions of $\lambda=(2,2)$, when $s=4$, and $\lambda=(2,2,2)$ when $s=6$, all Young tableaux with $s$ boxes and at most $d$ rows appear in the decomposition of $\g^{[s]}$ into irreducible $\GL_{d}$-modules.
\end{theorem}

We thus need to maximize $\frac{d_\lambda(k)}{d_\lambda(d)}$ over all diagrams with $s$ boxes, at most $d$ rows, and different from the above exceptions. For this we prove the following lemma:

\begin{lemma}\label{dominance}
Let $\mu$ and $\lambda$ be two Young diagrams with at most $d$ rows and having the same number of boxes. If $\mu$ can be obtained from $\lambda$ by moving some boxes downwards, then for any $k\geq d$,
$$\frac{d_\mu(k)}{d_\mu(d)} \geq \frac{d_\lambda(k)}{d_\lambda(d)}.$$
\end{lemma}

\begin{proof}Without loss of generality we may assume that $\mu$ is obtained from $\lambda$ by moving only one box downwards, i.e. $\mu_i=\lambda_i$, except for $\mu_{r}=\lambda_r-1$ and $\mu_{s}=\lambda_{s}+1$ for some indices $r<s$. Using Weyl's dimension formula:
$$d_\lambda(k)=\prod_{1\leq i<j \leq k}\frac{\lambda_i-\lambda_j+j-i}{j-i},$$
we get
$$\frac{d_\lambda(k)}{d_\mu(k)} \frac{d_\mu(d)}{d_\lambda(d)} = \prod_{d<j\leq k} \frac{\lambda_r +j-r}{\lambda_r +j - s}\frac{\lambda_{s} +j - r-1}{\lambda_{s} +j -r}.$$
But $\lambda_{s}\leq \lambda_r$, and hence each factor in the above product is at most $1$.
\end{proof}

Combined with Klyachko's theorem, this lemma allows us to compute $\widehat{\alpha}_k$. The Young diagram $\lambda$ present in $\g^{(s)}$ and achieving the maximal value in $(\ref{sigfor})$ is the diagram with $s$ boxes of the form $\lambda_0:=(2,1,\ldots,1)$, because we have assumed $s\leq d$. Using Weyl's dimension formula, it is easy to compute
$$d_{\lambda_0}(k)=(s-1){k+1 \choose s}.$$

When $k \geq d\geq s$, the relatively free Lie algebra $\cF_{k,\g}$ coincides with the free Lie algebra $\cF_{k,s}$ of step $s$. In particular, the growth exponent $\eta_{G}(k)$ is given by $(\ref{growthud})$. Given that $\widehat{\beta}_k = \widehat{\alpha}_k/\eta_{G}(k)$, this concludes the proof of Theorem \ref{freegroup}.

\bibliographystyle{alpha}
\bibliography{bibliography}

\end{document}